\documentclass[11pt]{amsart}

\usepackage{epigamath}

\usepackage[english]{babel}

\usepackage{makebox} 
\usepackage{bm} 
\usepackage{mathrsfs}
\usepackage{mathdots}
\usepackage{calc}
\usepackage{nicefrac} 
\usepackage{textcomp} 
\usepackage{mathtools} 
\usepackage{tikz}
\usepackage{tikz-cd} 
\usetikzlibrary{calc}
\usetikzlibrary{matrix,arrows}
\usetikzlibrary{decorations.markings}

\newtheorem{Thm}{Theorem}[section]              

\newtheorem{Lemma}[Thm]{Lemma}
\newtheorem{Cor}[Thm]{Corollary}

\newtheorem{thm}{Theorem}
\newtheorem{question}{Question}

\theoremstyle{definition}
\newtheorem{Rmk}[Thm]{Remark}

\newtheorem{Def}[Thm]{Definition}
\newtheorem{Ex}[Thm]{Example}
\newtheorem{Ax}[Thm]{Axiom}

\makeatletter \tikzcdset{
  open/.code     = {\tikzcdset{hook, circled};},
  closed/.code   = {\tikzcdset{hook, slashed};},
  open'/.code    = {\tikzcdset{hook', circled};},
  closed'/.code  = {\tikzcdset{hook', slashed};},
  circled/.code  = {\tikzcdset{markwith = {\draw (0,0) circle (.375ex);}};},
  slashed/.code  = {\tikzcdset{markwith = {\draw[-] (-.4ex,-.4ex) -- (.4ex,.4ex);}};},
  markwith/.code ={
    \pgfutil@ifundefined%
    {tikz@library@decorations.markings@loaded}%
    {\pgfutil@packageerror{tikz-cd}{You need to say %
      \string\usetikzlibrary{decorations.markings} to use arrows with markings}{}}{}%
    \pgfkeysalso{/tikz/postaction = {
      /tikz/decorate,
      /tikz/decoration={markings, mark = at position 0.5 with {#1}}}
    }
  },
}
\tikzcdset{
  shorten/.style={
    /tikz/shorten <=#1,
    /tikz/shorten >=#1},
} \makeatother

\renewcommand{\H}{\mathcal H}
\newcommand{\K}{\mathcal K}
\renewcommand{\L}{\mathbf L}
\newcommand{\A}{\mathbf A}
\renewcommand{\P}{\mathbf P}
\newcommand{\Z}{\mathbf Z}
\newcommand{\Q}{\mathbf Q}
\newcommand{\F}{\mathbf F}
\newcommand{\C}{\mathbf C}
\newcommand{\cs}{_{\text{\normalfont c}}}

\newcommand{\Hom}{\operatorname{Hom}}

\newcommand{\End}{\operatorname{End}}
\newcommand{\Ext}{\operatorname{Ext}}
\newcommand{\Mor}{\operatorname{Mor}}
\newcommand{\Tot}{\operatorname{Tot}}
\newcommand{\Spec}{\operatorname{Spec}}

\newcommand{\CH}{\operatorname{CH}}
\newcommand{\et}{_{\operatorname{\acute et}}}
\newcommand{\proet}{_{\operatorname{pro-\acute et}}}
\newcommand{\cl}{\operatorname{cl}}
\newcommand{\op}{^{\operatorname{op}}}
\newcommand{\even}{^{\operatorname{even}}}
\newcommand{\odd}{^{\operatorname{odd}}}
\newcommand{\tens}{\otimes}
\newcommand{\id}{\operatorname{id}}
\newcommand{\pt}{\operatorname{pt}}
\newcommand{\Corr}{\operatorname{Corr}}
\newcommand{\Rk}{\operatorname{Rk}}
\newcommand{\rk}{\operatorname{rk}}
\newcommand{\im}{\operatorname{im}}
\newcommand{\coker}{\operatorname{coker}}
\newcommand{\codim}{\operatorname{codim}}
\newcommand{\tr}{\operatorname{tr}}

\newcommand{\Dim}{\operatorname{Dim}}

\newcommand{\Kun}{\operatorname{K\ddot un}}
\newcommand{\Cl}{\operatorname{Cl}}
\newcommand{\Irr}{\operatorname{Irr}}
\newcommand{\cto}{\vdash}
\newcommand{\cyl}{\operatorname{cyl}}
\newcommand{\sing}{^{\operatorname{sing}}}

\newcommand{\rA}{\longrightarrow}
\newcommand{\Ra}{\Rightarrow}
\newcommand{\RA}{\Longrightarrow}

\newcommand{\LRa}{\Leftrightarrow}

\newcommand{\Ab}{\mathbf{Ab}}
\newcommand{\Set}{\mathbf{Set}}

\newcommand{\Topos}{\mathbf{Topos}}
\newcommand{\SHV}{\mathbf{Shv}}
\newcommand{\COMP}{\mathbf{Comp}}
\newcommand{\Ch}{\mathbf{Ch}}
\renewcommand{\Vec}{\mathbf{Vec}}
\newcommand{\gVec}{\mathbf{gVec}}
\newcommand{\gAlg}{\mathbf{gAlg}}
\newcommand{\Sch}{\mathbf{Sch}}
\newcommand{\SmPr}{\mathbf{SmPr}}
\newcommand{\SmPrVar}{\mathbf{SmPrVar}}
\newcommand{\M}{\mathbf{Mot}}
\newcommand{\aon}{^{\text{aon}}}

\newcommand{\punct}[1]{\makebox[0pt][l]{\,#1}} 

\newcommand{\fancyitem}[1]{%
 \renewcommand{\theenumi}{#1}
 \refstepcounter{enumi}
 \item[\rm #1]
}
\newcommand{\initiate}[1]{\setcounter{#1}{0}}

\EpigaVolumeYear{4}{2020} \EpigaArticleNr{16} \ReceivedOn{June 12,
2019}
\InFinalFormOn{August 1, 2020} \AcceptedOn{September 2, 2020}

\title{The equivalence of several conjectures on independence of $\ell$}
\titlemark{The equivalence of several conjectures on independence of $\ell$}

\author{Remy van Dobben de Bruyn}
\address{Department of Mathematics, Princeton University, Princeton, NJ 08544, United States of America}
\address{Institute for Advanced Study, Princeton, NJ 08540, United States of America}
\email{rdobben@math.princeton.edu}

\authormark{R. van Dobben de Bruyn}

\AbstractInEnglish{We consider several conjectures on the
independence of $\ell$ of the \'etale cohomology of (singular, open)
varieties over $\bar \F_p$. The main result is that independence of
$\ell$ of the Betti numbers $h^i\cs(X,\Q_\ell)$ for arbitrary
varieties is equivalent to independence of $\ell$ of homological
equivalence $\sim_{\text{hom},\ell}$ for cycles on smooth projective
varieties. We give several other equivalent statements. As a
surprising consequence, we prove that independence of $\ell$ of
Betti numbers for smooth quasi-projective varieties implies the same
result for arbitrary separated finite type $k$-schemes.}

\MSCclass{14F20 (primary); 14F30, 14C15, 14G15 (secondary)}

\KeyWords{Mathematics; algebraic geometry; arithmetic geometry;
\'etale cohomology; positive characteristic; motives; independence
of $\ell$}

\TitleInFrench{L'\'equivalence de plusieurs conjectures sur
l'ind\'ependance de $\ell$}

\AbstractInFrench{Nous consid\'erons plusieurs conjectures sur
l'ind\'ependance de $\ell$ pour la cohomologie \'etale des
vari\'et\'es (singuli\`eres, ouvertes) sur $\bar \F_p$. Le
r\'esultat principal est que l'ind\'ependance de $\ell$ des nombres
de Betti $h^i\cs(X,\Q_\ell)$ pour les vari\'et\'es arbitraires est
\'equivalente \`a l'ind\'ependance de $\ell$ de l'\'equivalence
homologique $\sim_{\text{hom},\ell}$ pour les cycles sur les
vari\'et\'es projectives lisses. Nous donnons plusieurs autres
\'enonc\'es \'equivalents. Comme cons\'equence surprenante, nous
prouvons que l'ind\'ependance de $\ell$ des nombres de Betti pour
les vari\'et\'es quasi-projectives lisses implique le m\^eme
r\'esultat pour les $k$-sch\'emas s\'epar\'es de type fini.}

\acknowledgement{Part of
this work was carried out with support of the Oswald Veblen Fund at
the Institute for Advanced Study.}

\begin{document}



\maketitle

\begin{prelims}

\DisplayAbstractInEnglish

\bigskip

\DisplayKeyWords

\medskip

\DisplayMSCclass

\bigskip

\languagesection{Fran\c{c}ais}

\bigskip

\DisplayTitleInFrench

\medskip

\DisplayAbstractInFrench

\end{prelims}


\newpage

\setcounter{tocdepth}{2}

\tableofcontents


\section*{Introduction}\addcontentsline{toc}{section}{Introduction}
Let $k$ be a field, and let $X$ be a $k$-variety. For every prime
number $\ell$ invertible in $k$, there is an associated \'etale
cohomology group $H^*\cs(X_{\bar k},\Q_\ell)$ defined using the
geometry of $\ell$-power degree covers of $X$. The main question we
want to consider is the following.

\begin{question}
Given a variety $X$ over a field $k$, is the dimension
$h^i\cs(X_{\bar k},\Q_\ell)$ of $H^i\cs(X_{\bar k},\Q_\ell)$
independent of the prime $\ell$?
\end{question}

If $k = \C$ and $X$ is smooth, this easily follows from the
functorial comparison isomorphisms
\cite[exp.~XI,~th.~4.4(iii)]{SGA4III}
\[
H^*\et(X,\Q_\ell) \cong H^*_{\operatorname{sing}}(X(\C),\Q)
\otimes_\Q \Q_\ell.
\]
The result for arbitrary $X$ over $\C$ can be deduced from this
using \makebox{hypercoverings}, cf.~\cite[6.2.8]{Hdg3}. The
Lefschetz principle proves the result for any field $k$ of
characteristic $0$, since \'etale cohomology is insensitive to
extensions of algebraically closed fields
\cite[exp.~XVI,~cor.~1.6]{SGA4III}.

On the other hand, if $k$ is finite and $X$ is smooth and proper,
then the Weil conjectures \cite{WeilI}, \cite{WeilII} imply that
$h^i\et(X_{\bar k},\Q_\ell)$ can be read off from the zeta function
of $X$, and thus does not depend on $\ell$. The question for
arbitrary $k$-varieties $X$ is a well-known open problem
\cite[p.~28,~(2a)]{Katz}, \cite[3.5(c)]{Ill}.

The homological standard conjecture \cite[\S4, Remarks~(3)]{Gro} is known
to imply the result in the following two cases:
\begin{enumerate}
\item[\rm (i)] $X$ is proper;
\item[\rm (ii)] $X$ is the complement of a simple normal crossings divisor $D$ in a smooth projective variety $\bar X$.
\end{enumerate}
Indeed, (ii) is explained in \cite[p.~28--29]{Katz}, and (i) is an
application of de Jong's alterations \cite{dJ}. Even assuming the
homological standard conjecture, the result for an arbitrary variety
$X$ does not seem to appear in the literature (although it may have
been known to experts). One cannot simply combine the arguments of
(i) and~(ii); see Remark \ref{Rmk combine proofs}.

We improve these conditional results in three ways:
\begin{itemize}
\item we replace the homological standard conjecture by a weaker assumption;
\item we prove independence of $\ell$ of $h^i\cs(X_{\bar k},\Q_\ell)$ for \emph{every} separated finite type $k$-scheme~$X$;
\item we prove a converse as well.
\end{itemize}

\begin{thm}\label{Thm independence}
Let $k$ be an algebraically closed field. If $k = \bar \F_p$, then
the following are equivalent:
\begin{enumerate}
\fancyitem{(\arabic{enumi})} \label{Item main 1} For every smooth
projective $k$-scheme $X$, the kernel of the cycle class map $\cl
\colon \CH^*_\Q(X) \to H^*(X,\Q_\ell)$ is independent of $\ell$;
\fancyitem{(\arabic{enumi})} \label{Item main 2} For all smooth
projective $k$-schemes $X$ and $Y$, any $\alpha \in \CH^*_\Q(X
\times Y)$, and any $i$, the rank of $\alpha_* \colon H^i(X,\Q_\ell)
\to H^*(Y,\Q_\ell)$\footnote{We do not write $H^i(Y,\Q_\ell)$
because $\alpha_*$ does not always take $H^i$ to $H^i$; see
Definition \ref{Def pushforward on Weil cohomology}.} is independent
of $\ell$;
\initiate{enumi} \fancyitem{(3\alph{enumi})} \label{Item main 3a}
For every separated finite type $k$-scheme $X$ and any $i$, the
dimension of $H^i\cs(X,\Q_\ell)$ is independent of $\ell$;
\fancyitem{(3\alph{enumi})} \label{Item main 3b} For every smooth
quasi-projective $k$-scheme $X$ and any $i$, the dimension of
$H^i\cs(X,\Q_\ell)$ is independent of $\ell$.
\end{enumerate}
Moreover, if these hold when $k = \bar \F_p$ for some prime $p$
(resp.\ for every prime $p$), then they hold over any algebraically
closed field of characteristic $p$ (resp.\ any algebraically closed
field).
\end{thm}

This result is given in Theorem \ref{Thm equivalent} and Remark
\ref{Rmk other fields} below. This gives many new angles to the
independence of $\ell$ question. The implication $\ref{Item main 3b}
\Ra \ref{Item main 3a}$ is particularly surprising; the proof goes
through \ref{Item main 1} and \ref{Item main 2}.

We also have an extension to crystalline cohomology. In fact, we
work with an arbitrary Weil cohomology theory (see Definition
\ref{Def Weil cohomology theory}) satisfying some additional axioms
(see Axiom \ref{Def additional axioms}), at the expense of
restricting to $k = \bar \F_p$. (Developing `Weil cohomology
theories with specialisation' would take us too far afield.)

However, our methods do not say anything about independence of
$\ell$ of the dimensions $h^i(X,\Q_\ell)$ of the (usual) cohomology
groups $H^i(X,\Q_\ell)$, except in the proper (resp.~smooth) case
where it coincides with (resp.~is dual to) compactly supported
cohomology.

The idea of $\ref{Item main 1} \Ra \ref{Item main 2}$ is that the
rank of a linear map $f \colon V \to W$ is the largest \makebox{$r
\in \Z_{\geq 0}$} such that $\bigwedge^r f \neq 0$. Although the
functors \makebox{$H^*(-,\Q_\ell) \colon \M_k \to \gVec$} do not
preserve wedge products (see Remark \ref{Rmk realisation functor}),
algebraicity of the K\"unneth projectors \cite{KM} decomposes a
cycle $\alpha \in \CH^*_\Q(X \times Y)$ as $\alpha_{\text{even}}
\oplus \alpha_{\text{odd}}$. Then $\bigwedge^r \alpha_{\text{even}}$
(resp.~$S^r \alpha_{\text{odd}}$) acts on cohomology as $\bigwedge^r
\alpha_{\text{even},*}$ (resp.~$\bigwedge^r \alpha_{\text{odd},*}$),
so the rank of the map $\alpha_* \colon H^*(X,\Q_\ell) \to H^*(Y,\Q_\ell)$
is determined by the vanishing or nonvanishing of
$\cl(\bigwedge^r\alpha_{\text{even}})$ and $\cl(S^r
\alpha_{\text{odd}})$ for various $r$.

To prove $\ref{Item main 2} \Ra \ref{Item main 3a}$, we use a
variant of the classical hypercovering argument
\cite[exp.~V$^{\text{bis}}$]{SGA4II}, \cite[\S5,6]{Hdg3}: if
$X_\bullet \to X$ is a proper hypercovering, then there is a
hypercohomology spectral sequence
\begin{equation}
E_1^{p,q} = H^q(X_p,\Q_\ell) \Ra H^{p+q}(X,\Q_\ell).\label{Eq
hypercovering spectral sequence}
\end{equation}
If each $X_p$ is smooth projective, then \eqref{Eq hypercovering
spectral sequence} degenerates on the $E_2$ page for weight reasons,
so $h^i(X,\Q_\ell)$ is determined by the ranks of the maps on the
$E_1$ page.

However, a proper hypercovering by smooth projective schemes can
only exist if $X$ is proper. In general, again using de Jong's
alterations \cite{dJ}, one can construct a proper hypercovering
$X_\bullet \to X$ where each $X_i$ is the complement of a simple
normal crossings divisor $Z_i$ in a smooth projective $k$-scheme
$\bar X_i$. There is a different spectral sequence
\cite[p.~28--29]{Katz} computing the compactly supported cohomology
of $X_i$ in terms of $\bar X_i$ and the components of $Z_i$; its
dual then computes the cohomology of $X_i$. However, if we then
compute \eqref{Eq hypercovering spectral sequence}, the purity
argument no longer applies.

Instead, we choose a compactification $X \to \bar X$ first, with
closed complement $V$, and we produce a morphism of simplicial
schemes $v_\bullet \colon V_\bullet \to \bar X_\bullet$, where
$V_\bullet$ (resp.~$\bar X_\bullet$) is a hypercovering of $V$
(resp.~$\bar X$). Then the simplicial mapping cone of $v_\bullet$
computes the compactly supported cohomology of $X$, by comparing the
long exact sequence for the mapping cone with that for the triple
$(\bar X,X,V)$. This allows us to apply the purity theorem as in the
argument above for $X$ proper, which finishes the proof of
$\ref{Item main 2} \Ra \ref{Item main 3a}$.

Finally, for the implication $\ref{Item main 3b} \Ra \ref{Item main
1}$ we prove that any cycle $\alpha \in \CH_d(X)$ can be written as
a difference $[Z_1] - [Z_2]$ with the $Z_i$ irreducible; see
Corollary \ref{Cor difference of irreducible cycles}. Letting $U$ be
the complement of $Z_1 \cup Z_2$, we relate the vanishing of
$\cl(\alpha)$ to the dimension of $H^{2d+1}\cs(U,\Q_\ell)$. There
are only two possible cases depending on whether $\cl(Z_1)$ and
$\cl(Z_2)$ are linearly independent or linearly dependent; in the
latter case the linear relation is determined by intersection
numbers.

\subsection*{Outline of the paper}
In Section \ref{Sec motives} we give a brief review of Weil
cohomology theories (Definition \ref{Def Weil cohomology theory})
and pure motives (Definition \ref{Def Chow motive}). Section
\ref{Sec simplicial} contains a review of simplicial schemes and
mapping cones, which play a role\footnote{Weizhe Zheng provided an
easier argument without mapping cones; see Remark \ref{Rmk WZ}. We
kept the original argument as it proves more, and the techniques
might be useful elsewhere.} in $\ref{Item main 2} \Ra \ref{Item main
3a}$. In Section \ref{Sec all-or-nothing} we state the additional
axioms on our Weil cohomology theory for the arguments to work; see
Axiom \ref{Def additional axioms}.

The main theorem will be stated in Section \ref{Sec independence}
(see Theorem \ref{Thm equivalent}). We then proceed to prove the
implications of Theorem \ref{Thm equivalent} as outlined in the
introduction above, in the following cyclic order:
\[
\ref{Item main 1} \RA \ref{Item main 2} \RA \ref{Item main 3a} \RA
\ref{Item main 3b} \RA \ref{Item main 1}.
\]
The implication $\ref{Item main 3a} \Ra \ref{Item main 3b}$ is
trivial; each of the others will take up one section (Section
\ref{Sec Cycle classes and ranks}, Section \ref{Sec Ranks and
dimensions}, and Section \ref{Sec Dimensions and cycle class maps}
respectively).

\phantomsection
\subsection*{Notation and conventions}\label{Sec Notation}
If $k$ is a field, then a \emph{$k$-variety} will mean a finite
type, separated, geometrically integral $k$-scheme. A pair $(X,H)$
is called a projective $k$-scheme if $X$ is a projective $k$-scheme
and $H$ a very ample divisor on $X$.

In the main theorems, the base field $k$ will be assumed
algebraically closed, because standard references on Weil cohomology
theories have this running assumption, and establishing the general
framework would take us too far astray.

The category of smooth projective $k$-varieties will be denoted by
$\SmPrVar_k$, and the category of smooth projective $k$-schemes will
be denoted by $\SmPr_k$. The latter can be obtained from the former
as the category of formal finite coproducts (if $k$ is algebraically
closed), cf.~Example \ref{Ex Coprod}. The category of Chow motives
is denoted by $\M_k$; its definition will be recalled in Definition
\ref{Def Chow motive}. Morphisms in this category are typically
denoted by $\alpha \colon X \cto Y$.

If $K$ is a field, then $\Vec_K$ denotes the category of $K$-vector
spaces, $\gVec_K$ the category of $\Z$-graded $K$-vector spaces, and
$\gAlg_K$ the category of $\Z$-graded (unital, associative)
$K$-algebras. The objects of $\gAlg_K$ we encounter will always be
graded-commutative and vanish in negative degrees.

We write $\Ab(\mathscr C)$ for the category of abelian group objects
in a category $\mathscr C$ with finite products. All
topoi\footnote{We work with topoi instead of sites because they have
better formal properties, but all arguments could also be carried
out using Grothendieck (pre)topologies. We ignore set-theoretic
issues; they can for instance be dealt with using universes.} will
be Grothendieck topoi, i.e.~the topos of sheaves (of sets) on a
small category with a Grothendieck topology (or pretopology). We
write $\Topos$ for the (strict) $2$-category of topoi, whose objects
are topoi, whose $1$-morphisms are (geometric) morphisms of topoi,
and whose $2$-morphisms are natural transformations between the
inverse image functors (equivalently, between the direct image
functors).

We write $\SHV$ for the (strict) $2$-category whose objects are
pairs $(X,\mathscr F)$ of a topos $X$ with an abelian object
$\mathscr F$ in $X$, whose $1$-morphisms $(X,\mathscr F) \to
(Y,\mathscr G)$ are pairs $(f,\phi)$ of a $1$-morphism $f \colon X
\to Y$ and a morphism $\phi \colon f^* \mathscr G \to \mathscr F$ of
abelian objects, and whose $2$-morphisms $(f,\phi) \to (g, \psi)$
are given by natural transformations $\eta \colon f^* \Ra g^*$ such
that $\psi \circ \eta_{\mathscr G} = \id_{\mathscr F} \circ \phi$.
We think of it as a ``fibred $2$-category'' $\SHV \to \Topos$, whose
fibre above the topos $X$ is $\Ab(X)\op$ (with only identity
$2$-morphisms). In a similar way, we define a category $\COMP$ of
pairs $(X,K)$ of a topos $X$ with a complex $K$ of abelian objects
on $X$.

\subsection*{Acknowledgements}
{\small This paper grew out of a conversation with Ashwin Deopurkar
on the existence of exterior powers in the category of pure motives.
I also thank Raymond Cheng, Johan Commelin, and Johan de Jong for
helpful conversations. I am particularly grateful to Weizhe Zheng
for providing an alternative and easier proof of $\ref{Item main 2}
\Ra \ref{Item main 3a}$; see Remark \ref{Rmk WZ}. Finally, I thank
the referee for many helpful suggestions and corrections. }

\numberwithin{equation}{section}

\section{Pure motives and Weil cohomology theories}\label{Sec motives}
This is a review of the theory of pure motives, cf.\ e.g.\ Kleiman \cite{KleAG}, Jannsen \cite{Jan}, or Scholl's excellent survey \cite{Sch}. We also give a brief review of Weil cohomology theories; see \cite[\S 3]{KleMot} for more details. 
%
Following standard references, we will assume that $k$ is
algebraically closed. Our setup is slightly more general than
\cite[\S 3]{KleMot}, in that we allow smooth projective $k$-schemes
with multiple components.

\begin{Def}\label{Def Weil cohomology theory}
Let $k$ be an algebraically closed field, and let $K$ be a field of
characteristic $0$. A \emph{Weil cohomology theory} is a functor $H
\colon \SmPr_k\op \to \gAlg_K$ satisfying the following axioms.
\renewcommand{\theenumi}{(W\arabic{enumii})}
\begin{enumerate}
\fancyitem{\rm (W\arabic{enumi})} \label{Ax W1} Each $H^i(X)$ is
finite-dimensional and vanishes for $i < 0$ and $i > 2 \dim X$;
\fancyitem{\rm (W\arabic{enumi})} \label{Ax W2} There is a trace map
$\tr_X \colon H^{2\dim X}(X) \to K$ that is an isomorphism if $X$ is
irreducible, and takes $1$ to $1$ if $X = \Spec k$. If all
components of $X$ have the same dimension $d$, then the natural pairing
$H^i(X) \times H^{2d-i}(X) \to K$ is perfect; \fancyitem{\rm
(W\arabic{enumi})} \label{Ax W3} The projections induce an
isomorphism $H^*(X) \otimes_K H^*(Y) \stackrel{\!\sim}\to H^*(X
\times Y)$; \fancyitem{\rm (W\arabic{enumi})} \label{Ax W4} There
are cycle class maps $\cl \colon \CH^i_\Q(X) \to H^{2i}(X)$. It is a
ring homomorphism functorial for pullback and pushforward, where
pushforward for $H$ is defined using (W2); \fancyitem{\rm
(W\arabic{enumi})} \label{Ax W5} The weak Lefschetz theorem holds;
\fancyitem{\rm (W\arabic{enumi})} \label{Ax W6} The hard Lefschetz
theorem holds; \fancyitem{\rm (W\arabic{enumi})} \label{Ax W7} $H$
preserves products, i.e. $H(X) \cong \prod_i H(X_i)$ if $X =
\coprod_i X_i$. Moreover, this isomorphism identifies $\tr_X$ with
$\sum_i \tr_{X_i}$.
\end{enumerate}
\end{Def}

\begin{Ex}\label{Ex etale}
For every prime $\ell$ invertible in $k$, the $\ell$-adic \'etale
cohomology gives a Weil cohomology theory \cite{SGA4III,WeilI,
WeilII}.
\end{Ex}

\begin{Ex}\label{Ex crystalline}
For a perfect field $k$ of positive characteristic $p$ with Witt
ring $W(k)$ with field of fractions $K$, crystalline cohomology
$H^i_{\text{cris}}(X/K)$ is a Weil cohomology theory \cite{BerCris},
\cite[Corollary~1(2)]{KM}, \cite{Gros}, \cite{GilMes}. See also
\cite{IllCrys} for an expository account and additional references.
\end{Ex}

Let $\sim$ be an adequate equivalence relation (cf.\ \cite{Sam})
finer than homological equivalence for every Weil cohomology theory,
e.g. $\sim$ is rational, algebraic, or smash-nilpotent equivalence.
For a variety $X$ and $i \in \Z_{\geq 0}$, we will write $\CH^i(X)$
for algebraic cycles of codimension $i$ modulo $\sim$, and
$\CH^i_\Q(X)$ for $\CH^i(X) \otimes \Q$.

We will omit further mention of $\sim$ unless it plays a role in the
arguments.

\begin{Def}\label{Def correspondences}
Let $X$ and $Y$ be smooth projective $k$-schemes, and assume that
$X$ has connected components $X_1, \ldots, X_m$ of dimensions $d_1,
\ldots, d_m$ respectively. Then the group of \emph{correspondences
of degree $r$ (modulo $\sim$)} from $X$ to $Y$ is
\[
\Corr^r(X,Y) := \bigoplus_{i=1}^m \CH^{d_i+r}_\Q(X_i \times Y).
\]
An element $\alpha \in \Corr^r(X,Y)$ is a \emph{correspondence} from
$X$ to $Y$, and is denoted
$\alpha \colon X \cto Y$.
%
There is a natural composition of correspondences:
\begin{align*}
\Corr^r(X,Y) \times \Corr^s(Y,Z) &\to \Corr^{r+s}(X,Z)\\
(\alpha, \beta) &\mapsto \beta \circ \alpha.
\end{align*}
If the degree $r$ in the superscript is omitted, it will be assumed
$0$.
\end{Def}

\begin{Def}\label{Def Chow motive}
The \emph{category of Chow motives (modulo $\sim$)} is the category
whose objects are triples $(X,p,m)$, where $X$ is a smooth
projective $k$-scheme, $p \in \Corr(X,X)$ a projector (i.e. $p^2 =
p$), and $m \in \Z$ an integer. Morphisms $(X,p,m) \cto (Y,q,n)$
from $(X,p,m)$ to $(Y,q,n)$ are given by
\begin{align*}
\Hom\big((X,p,m),(Y,q,n)\big) &= q\Corr^{n-m}(X,Y)p\\
&= \{\alpha \in \Corr^{n-m}(X,Y)\ |\ \alpha p = \alpha = q \alpha\}.
\end{align*}
We denote the category of Chow motives by $\M_k$. The motive $(\Spec
k,\id,-1)$ is called the \emph{Lefschetz motive}, and is denoted by
$\L$. We write $\L^n$ for $(\Spec k,\id,-n)$, which can also be
defined as $\L^{\otimes n}$ using the tensor product of Remark
\ref{Rmk tensor product}.
\end{Def}

\begin{Rmk}\label{Rmk motives contravariant}
There is a functor $\SmPr_k\op \to \M_k$ associating to every smooth
projective $k$-scheme $X$ the motive $(X,\id,0)$, and to every map
$f \colon X \to Y$ the (class of the) graph $\Gamma_f \in
\Corr(Y,X)$.
\end{Rmk}

\begin{Def}\label{Def pushforward on Weil cohomology}
If $H$ is a Weil cohomology theory with coefficient field $K$, then
Poincar\'e duality gives an isomorphism
\begin{align}
H^*(X) &\rA H^*(X)^\vee\label{Eq Poincare duality}\\
v &\longmapsto \left(w \mapsto \int v \smile w\right).\nonumber 
\end{align}
Together with the K\"unneth formula this gives isomorphisms
\[
H^*(X \times Y) \cong H^*(X) \tens H^*(Y) \cong H^*(X)^\vee \tens
H^*(Y) = \Hom(H^*(X),H^*(Y)).
\]
If $\alpha \in \Corr^r(X,Y)$, then under these isomorphisms
$\cl(\alpha)$ induces a pushforward
\begin{equation}
\alpha_* \colon H^i(X) \to H^{i+2r}(Y).\label{Eq pushforward on Weil
cohomology}
\end{equation}
In particular, a projector $p \in \Corr(X,X)$ induces a projector on
$H^*$, and we extend $H$ to a functor $\M_k \to \gVec_K$ by setting
\[
H^*(X,p,m) := pH^*(X)[2m],
\]
where for a graded vector space $V = \bigoplus_i V^i$, we set
$V^i[m] = V^{i+m}$.\pagebreak

Given a morphism $\alpha \colon (X,p,m) \cto (Y,q,n)$, we define
$H\alpha$ to be the graded map given by the pushforward $\alpha_*
\colon pH^{i+2m}(X) \to qH^{i+2n}(Y)$ as in \eqref{Eq pushforward on
Weil cohomology}.

Some Weil cohomology theories have further structure (e.g.\ a Hodge
structure or a Galois action), and these structures are typically
preserved by pushforward along cycles (if everything is given the
correct `Tate twist'). We will consider this additional structure
understood, and we will not use separate notation for the
corresponding enriched functor.
\end{Def}

\begin{Ex}
The cohomology $H^i(\L)$ of the Lefschetz motive is $0$ if $i \neq
2$ and one-dimensional if $i = 2$. It's often thought of as the
compactly supported cohomology of $\A^1$ (or the reduced cohomology
of $\P^1$), and is equipped with the corresponding Galois action or
Hodge structure.
\end{Ex}

\begin{Rmk}\label{Rmk tensor product}
The category $\M_k$ has a tensor product given by
\[
(X,p,m) \tens (Y,q,n) = (X \times Y, p \tens q, m+n).
\]
Thus, we also get symmetric and alternating products $S^n$ and
$\bigwedge^n$ by considering the projectors $\frac{1}{n!}
\sum_\sigma \sigma$ and $\frac{1}{n!} \sum_\sigma (-1)^\sigma
\sigma$ respectively on $X^n$.
\end{Rmk}

\begin{Rmk}\label{Rmk realisation functor}
If $H$ is a Weil cohomology theory with coefficient field $K$, then
the functor $H \colon \M_k \to \gVec_K$ is a tensor functor if we
equip $\gVec_K$ with the tensor product as in super vector spaces:
on objects, it is given by the usual graded tensor product, but the
swap is given by
\begin{align*}
\tau_{\substack{\\V,W}} \colon V \tens W &\to W \tens V\\
v_i \tens w_j &\mapsto (-1)^{ij} w_j \tens v_i,
\end{align*}
for homogeneous elements $v_i \in V^i$, $w_j \in W^j$.
To see that this makes $H$ into a tensor functor, note that the
K\"unneth isomorphism is given by the map
\begin{align*}
H^*(X) \tens H^*(Y) &\to H^*(X\times Y)\\
\alpha \tens \beta &\mapsto \pi_X^* \alpha \smile \pi_Y^* \beta,
\end{align*}
which under swapping $X$ and $Y$ picks up a factor $(-1)^{\deg
\alpha \deg \beta}$.
\end{Rmk}

\begin{Rmk}\label{Rmk super}
In particular, if $H^*(X) = H\even \oplus H\odd$, then
\[
H^*\Big(S^r X\Big) = S^r\big(H^*(X)\big) = \bigoplus_{i+j=r} S^i
H\even \tens {\textstyle\bigwedge\nolimits}^j H\odd,
\]
and conversely
\[
H^*\Big({\textstyle\bigwedge\nolimits}^r X\Big) =
{\textstyle\bigwedge\nolimits}^r\big(H^*(X)\big) = \bigoplus_{i+j=r}
{\textstyle\bigwedge\nolimits}^i H\even \tens S^j H\odd.
\]
Indeed, because of the sign in $\tau_{V,W}$, the symmetriser and
antisymmetriser get swapped in odd degree.
\end{Rmk}

\begin{Rmk}\label{Rmk biproducts}
The category $\M_k$ has binary biproducts; for example
\[
(X,p,n) \oplus (Y,q,n) = (X \amalg Y, p \oplus q, n).
\]
One can construct $(X,p,m) \oplus (Y,q,n)$ when $m \neq n$ as well,
by replacing $X$ by $X \times \P^{n-m}$ or $Y$ by $Y \times
\P^{m-n}$, using that $\L^d = (\pt,\id,-d)$ is a summand of
$(\P^d,\id,0)$ for $d \geq 0$ (see e.g.\ \cite[1.13]{Sch}), and
using the tensor product of Remark \ref{Rmk tensor product} to
reduce to the case $m = n$.

However, the category $\M_k$ is not in general abelian, and in fact
for our choices of $\sim$ this is either false or open; see for
example \cite[Corollary~3.5]{Sch} for the case where $\sim$ is rational
equivalence and $k$ is not isomorphic to $\bar \F_p$.
\end{Rmk}

\begin{Lemma}\label{Lem vanish}
Let $X$ and $Y$ be smooth projective schemes over a field $k$, and
let $\alpha \in \Corr(X,Y)$. Then the map $\alpha_* \colon H^*(X)
\to H^*(Y)$ is $0$ if and only if $\cl(\alpha) = 0 \in H^*(X \times
Y)$.
\end{Lemma}

\begin{proof}
This is clear from the definition of $\alpha_*$, cf.\ Definition
\ref{Def pushforward on Weil cohomology}.
\end{proof}

\begin{Lemma}\label{Lem trace}
Let $X$ be an $n$-dimensional smooth projective $k$-scheme, and let
$\alpha \in \Corr(X,X)$. Then
\[
\sum_{i = 0}^{2n} (-1)^i \tr\left(\alpha_*\big|_{H^i(X)}\right) =
\alpha \cdot [\Delta_X],
\]
where $\Delta_X \subseteq X \times X$ is the diagonal. In particular, it is a rational number that does not depend on $H$. 
\end{Lemma}

\begin{proof}
The first statement is well-known, and the second follows.
\end{proof}

\begin{Cor}\label{Cor characteristic polynomial}
Let $X$ be a smooth projective $k$-variety such that $\Kun(X)$ holds
(see Definition \ref{Def conjectures}). Let $\alpha \in \Corr(X,X)$.
Then the characteristic polynomial of $\alpha$ on $H^i(X)$ is in
$\Q[t]$, and is independent of the Weil cohomology theory $H$.
\end{Cor}

\begin{proof}
One easily checks that the coefficients of the characteristic
polynomial
$$P_A(t) = \det(t \cdot I - A) = t^n + c_{n-1}t^{n-1} +\ldots + c_1 t + c_0$$
of an endomorphism $A$ on an $n$-dimensional
vector space $V$ are given by
\[
c_j = (-1)^{n-j}\tr\left({\textstyle\bigwedge}^{n-j}A\right).
\]
Hence, if $p_i \in \Corr(X,X)$ denotes the $i$-th K\"unneth
projector, then applying Lemma \ref{Lem trace} to
$\bigwedge^{n-j}(p_i \circ \alpha)$ gives the result.
\end{proof}

\section{Simplicial topoi and mapping cones}\label{Sec simplicial}
We will use the following (possibly non-standard) terminology:

\begin{Def}
Let $D$ be a small category, and let $\mathscr C$ be a category.
Then a \emph{$D$-object} in $\mathscr C$ is a functor $D\op \to
\mathscr C$. If $D$ is the category $\Delta_+$ of finite (totally)
ordered sets with monotone maps (resp.\ the subcategory $\Delta$ of
nonempty objects), then a $D$-object is an \emph{augmented
simplicial object} (resp.\ \emph{simplicial object}) of $\mathscr
C$.

If $X_\bullet$ is an (augmented) simplicial object, then $X_n$
denotes the value of $X_\bullet$ on the set $[n] = \{0,\ldots,n\}$.
Giving an augmented simplicial object $(X_n)_{n \geq -1}$ is
equivalent to giving a simplicial object $X_\bullet = (X_n)_{n \geq
0}$ together with a map $X_\bullet \to X_{-1}$ to the constant
simplicial object with value $X_{-1}$.
\end{Def}

\begin{Rmk}\label{Rmk simplicial topos}
If $\mathscr C$ is a $2$-category (e.g.~$\Topos$ or $\SHV$; see
\hyperref[Sec Notation]{Notation}), then the correct definition of a
$D$-object is a pseudofunctor $X_\bullet \colon D\op \to \mathscr
C$. By the `Grothendieck construction', this should correspond to
some sort of fibred object. For example, a $D$-topos corresponds to
a functor $F \colon X \to D$ that is a fibration and cofibration,
such that each fibre $X_i$ is a topos, and for any morphism $\phi
\colon i \to j$ in $D$ the pair $(f_*,f^*) = (\phi^*,\phi_*)$ is a
morphism of topoi $f \colon X_j \to X_i$. This is the definition of
$D$-topos in \cite[exp.~V$^{\text{bis}}$,~d\'ef.~1.2.1]{SGA4II}.

By
\cite[exp.~V$^{\text{bis}}$,~d\'ef.~1.2.8~and~prop.~1.2.12]{SGA4II},
this gives rise to a total topos $\mathbf \Gamma(X)$. Abelian
objects in $\mathbf \Gamma(X)$ correspond to enrichments of the
pseudofunctor $X_\bullet \colon D\op \to \Topos$ through $D\op \to
\SHV$, and similarly for complexes.
\end{Rmk}

\begin{Rmk}\label{Rmk change of index}
Let $f \colon D' \to D$ be a functor, and let $X \to D$ a $D$-topos.
Then the pullback $D' \times_D X \to D'$ is a $D'$-topos, and $f^*
\colon \mathbf \Gamma(X) \to \mathbf \Gamma(D' \times_D X)$ has left
and right adjoints $f_!$ and $f^*$ \cite[exp.~V$^{\text{bis}}$,\
prop.~1.2.9]{SGA4II}; in particular, $(f_*,f^*) \colon \mathbf
\Gamma(D' \times_D X) \to \mathbf \Gamma(X)$ is a morphism of topoi.

For an inclusion $e_d \colon d \to D$ of an object $d \in D$, we
write $(-)|_{X_d}$ instead of $e_d^*$. In this case
\cite[exp.~V$^{\text{bis}}$,~cor.~1.2.11~and~prop.~1.3.7]{SGA4II}
show that $e_{d,!}$ is given on abelian objects by
\[
e_{d,!}(\mathscr F)(d') = \bigoplus_{\alpha \colon d \to d'}
\alpha_*(\mathscr F).
\]
Note that $e_{d,!}$ is exact, since $\alpha_*$ is exact (since
$(f_*,f^*) = (\alpha^*,\alpha_*)$ is a morphism of topoi), direct sums are exact
\cite[exp.~II,~prop.~6.7]{SGA4I}, and
exactness in $\Ab(\mathbf \Gamma(X))$ is determined pointwise.
\end{Rmk}

\begin{Rmk}\label{Rmk simplicial notation}
If $D = \Delta$, then following the notation of simplicial
topological spaces we will write $\Ab(X_\bullet)$ (resp.\
$\Ch(X_\bullet)$) for $\Ab(\mathbf \Gamma(X))$ (resp.\ $\Ch(\mathbf
\Gamma(X))$) and write $H^i(X_\bullet, K)$ for $H^i(\mathbf
\Gamma(X),K)$. For a bounded below complex\footnote{Note that a
complex $K$ on $X_\bullet$ is bounded below iff the complexes on the
components $X_i$ are bounded below \emph{uniformly in $i$}.} $K$ on
$X_\bullet$, there is a spectral sequence
\[
E_1^{p,q} = H^q\big(X_p,K|_{X_p}\big) \Ra H^{p+q}(X_\bullet,K),
\]
whose $E_1$ page is the alternating face complex on the cosimplicial
abelian groups $H^q(X_*,K|_{X_*})$; see for example
\cite[exp.~V$^{\text{bis}}$,~cor.~2.3.7 and~2.3.9]{SGA4II}. For
computation, it's useful to have a refined version of this
statement:
\end{Rmk}

\begin{Lemma}\label{Lem componentwise acyclic}
Let $X_\bullet$ be a simplicial topos, and let $K$ be a bounded
below complex on $X_\bullet$ with $H^i(X_n, K^j|_{X_n}) = 0$
for all $j \in \mathbf Z$, $n \in \mathbf N$, and $i > 0$. Then the
complex
\[
\Tot\left(\Gamma\big(X_*, K^*\big|_{X_*} \big)\right)
\]
computes $R\Gamma(X_\bullet,K)$.
\end{Lemma}

\begin{proof}
Let $\Set \times \Delta$ be the constant simplicial topos, so
$\mathbf \Gamma(\Set \times \Delta) = [\Delta,\Set]$. Since $\Set$
is terminal in $\Topos$ \cite[exp.~IV,~\S4.3]{SGA4I}, the constant
simplicial topos is the terminal simplicial topos, so we get a
functor $\theta \colon X_\bullet \to \Set \times \Delta$.

Again since $\Set$ is terminal, the terminal morphism $\Gamma \colon
\mathbf \Gamma(X) \to \Set$ factors as
\begin{equation*}
\begin{tikzcd}
\mathbf \Gamma(X) \ar{r}{\mathbf \Gamma(\theta)} &
\left[\Delta,\Set\right] \ar{r}{\varepsilon} & \Set\punct{.}
\end{tikzcd}
\end{equation*}
On the level of derived categories of abelian objects, we get
\begin{equation*}
\begin{tikzcd}[column sep=2.8em]
D^+(X_\bullet) \ar{r}{R\mathbf \Gamma(\theta)_*} &
D^+\big(\left[\Delta,\Ab\right]\big) \ar{r}{R\varepsilon_*} &
D^+(\Ab)\punct{.}
\end{tikzcd}
\end{equation*}
By \cite[exp.~V$^{\text{bis}}$,~prop.~1.3.7~and~cor.~1.3.12]{SGA4II}
we have
\[
\Big(R^i\mathbf \Gamma(\theta)_* K^j\Big)_n =
R^i\theta_{n,*}\big(K^j\big|_{X_n}\big) =
H^i\left(X_n,\big(K^j\big|_{X_n}\big)\right) = 0
\]
for all $n \in \mathbf N$, all $j \in \mathbf Z$, and all $i > 0$.
Thus, all $K^j$ are acyclic for $R\mathbf \Gamma(\theta)_*$, so
\[
R\mathbf \Gamma(\theta)_* K = \mathbf \Gamma(\theta)_* K.
\]
The result follows since $R\varepsilon_*$ is computed by the
totalisation of the alternating face complex
\cite[exp.~V$^{\text{bis}}$,~cor.~2.3.6]{SGA4II}.
\end{proof}

\begin{Def}\label{Def simplicial cone}
Let $\mathscr C$ be a category with finite coproducts and a terminal
object $*$, and let $f_\bullet \colon Y_\bullet \to X_\bullet$ be a
morphism of simplicial objects in $\mathscr C$. Then the
\emph{mapping cone} $C_\bullet(f)$ of $f_\bullet$ is the simplicial
object in $\mathscr C$ constructed as the pushout of the diagram
\begin{equation*}
\begin{tikzcd}[column sep=.5em,row sep=.7em]
 & Y_\bullet\ar[end anchor=north east]{ld}\ar[end anchor=north west]{rd}{i_0} & & Y_\bullet\ar[end anchor=north east]{ld}[swap]{i_1}\ar[end anchor=north west]{rd}{f_\bullet} & \\
* & & Y_\bullet \times \Delta[1] & & X_\bullet\punct{,}
\end{tikzcd}
\end{equation*}
where $Y_\bullet \times \Delta[1]$ is the simplicial object of
$\mathscr C$ defined in \cite[Tag
\href{https://stacks.math.columbia.edu/tag/017C}{017C}]{Stacks}. The
mapping cone satisfies
\begin{equation}
C_n(f) = * \amalg \underbrace{Y_n \amalg \ldots \amalg Y_n}_n
\amalg\ X_n.\label{Eq mapping cone}
\end{equation}
A slightly different construction is given in \cite[6.3.1]{Hdg3},
but the two are homotopy equivalent as simplicial objects in
$\mathscr C$. Similarly, the \emph{mapping cylinder}
$\cyl_\bullet(f)$ is the pushout of
\begin{equation*}
\begin{tikzcd}[column sep=0em,row sep=1.7em]
Y_\bullet\ar{d}[swap]{i_1}\ar{r}{f_\bullet} & X_\bullet \\
Y_\bullet \times \Delta[1]\punct{.} &
\end{tikzcd}
\end{equation*}
It satisfies
\begin{equation}
\cyl_n(f) = \underbrace{Y_n \amalg \ldots \amalg Y_n}_{n+1} \amalg\
X_n,\label{Eq mapping cylinder}
\end{equation}
and the inclusion $X_\bullet \to \cyl_\bullet(f)$ and projection
$\cyl_\bullet(f) \stackrel{f_\bullet}\to X_\bullet \times \Delta[1]
\to X_\bullet$ are homotopy inverses as simplicial object in
$\mathscr C$.
\end{Def}

\begin{Ex}
Let $f_\bullet \colon Y_\bullet \to X_\bullet$ be a morphism of
simplicial topoi, let $K$ and $L$ be abelian objects
(resp.~complexes) on $X_\bullet$ and $Y_\bullet$ respectively, and
let $\phi \colon f_\bullet^* K \to L$ be a morphism. Then we may
view $(f_\bullet,\phi)$ as a morphism $(Y_\bullet,L) \to
(X_\bullet,K)$ of simplicial objects in $\SHV$ (resp.~$\COMP$).
Thus, there is an abelian object (resp.~complex) $C_\bullet(\phi)$
on $C_\bullet(f)$ such that $(C_\bullet(f),C_\bullet(\phi))$ is the
mapping cone in $\SHV$ (resp.~$\COMP$).
\end{Ex}

\begin{Rmk}\label{Rmk terminal site}
The terminal object in $\Topos$ is $\Set$ \cite[exp.~IV,~\S
4.3]{SGA4I}, whose abelian objects are just abelian groups.
Moreover, $2$-coproducts in $\Topos$ are given by the product of
categories \cite[exp.~IV, exercice~8.7(bc)]{SGA4I}. The terminal object
in $\SHV$ or $\COMP$ is the pair $(\Set,0)$.
Thus, in the description of \eqref{Eq mapping cone}, the complex
$C_n(\phi)$ on $C_n(f)$ equals $0$ on $* = \Set$, equals $L|_{Y_n}$
on each of the components $Y_n$, and equals $K|_{X_n}$ on the
component $X_n$.
\end{Rmk}

\begin{Lemma}\label{Lem distinguished triangle}
Let $X_\bullet$ and $Y_\bullet$ be simplicial topoi, and let
$f_\bullet \colon Y_\bullet \to X_\bullet$ be a morphism of
simplicial topoi. Let $K$ and $L$ be bounded below complexes on
$X_\bullet$ and $Y_\bullet$ respectively, and let $\phi \colon
f_\bullet^* K \to L$ be a morphism. Then there is a canonical
distinguished triangle
\begin{equation*}
R\Gamma\big(C_\bullet(f),C_\bullet(\phi)\big) \to
R\Gamma\big(X_\bullet,K\big) \to R\Gamma\big(Y_\bullet,L\big) \to
R\Gamma\big(C_\bullet(f),C_\bullet(\phi)\big)[1]
\end{equation*}
in $D^+(\Ab)$, functorial in the morphism $(f,\phi) \colon
(Y_\bullet,L) \to (X_\bullet,K)$ of simplicial objects in $\COMP$.
\end{Lemma}

In the setting of simplicial topological spaces, this is
\cite[6.3.3]{Hdg3}.

\begin{proof}
View $(\phi \colon f_\bullet^* K \to L)$ as a bounded below complex
$\mathbf K$ on the $(\leftarrow)$-topos $\mathbf X = (\mathbf
\Gamma(Y) \to \mathbf \Gamma(X))$. Let $\mathbf K \to \mathbf I$ be
an injective resolution. This means that $\mathbf I = (\psi \colon
f_\bullet^* I \to J)$ for bounded below complexes of injectives $I$
and $J$ on $X_\bullet$ and $Y_\bullet$ respectively, together with a
commutative square
\begin{equation*}
\begin{tikzcd}
f_\bullet^*K \ar{r}{f_\bullet^*i}\ar{d}[swap]{\phi} & f_\bullet^*I \ar{d}{\psi}\\
L \ar{r}[swap]{j} & J\punct{.}
\end{tikzcd}
\end{equation*}
Since $e_{X,!}$ and $e_{Y,!}$ are exact by Remark \ref{Rmk change of
index}, we conclude that $I = e_X^* \mathbf I$ and $J = e_Y^*
\mathbf I$ are still injective
\cite[Tag~\href{https://stacks.math.columbia.edu/tag/015Z}{015Z}]{Stacks}.
Since $e_X^*$ and $e_Y^*$ are exact, $i \colon K \to I$ and $j
\colon L \to J$ are still resolutions. Similarly, $K|_{X_n} \to
I|_{X_n}$ and $L|_{Y_n} \to J|_{Y_n}$ are injective resolutions for
any $n \in \mathbf N$.  Finally, functoriality of the mapping cone
and mapping cylinder gives resolutions
\begin{align*}
C_\bullet(\phi) &\to C_\bullet(\psi) \\
\cyl_\bullet(\phi) &\to \cyl_\bullet(\psi),
\end{align*}
where all terms $C_n(\psi)$ and $\cyl_n(\psi)$ are injective (but we
do not know whether $C_\bullet(\psi)$ and $\cyl_\bullet(\psi)$
themselves are injective).

From the descriptions of \eqref{Eq mapping cone}, \eqref{Eq mapping
cylinder}, and Remark \ref{Rmk terminal site}, we see that the maps
$Y_\bullet \stackrel{i_0}\to \cyl_\bullet(f) \to C_\bullet(f)$
induce a termwise split exact sequence
\begin{equation}
0 \to \Tot\Big(\Gamma\big(C_*,C_\bullet^*(\psi)\big)\Big) \to
\Tot\Big(\Gamma\big(\cyl_*,\cyl_\bullet^*(\psi)\big)\Big) \to
\Tot\big(\Gamma(Y_*,J^*)\big) \to 0.\label{Dia termwise split}
\end{equation}
The homotopy equivalence $(\cyl_\bullet,\cyl_\bullet(\psi)) \simeq
(X_\bullet,I)$ induces a chain homotopy equivalence
\[
\Tot\Big(\Gamma\big(\cyl_*,\cyl_\bullet^*(\psi)\big)\Big) \simeq
\Tot\big(\Gamma(X_*,I^*)\big)
\]
\cite[Tags~\href{https://stacks.math.columbia.edu/tag/019M}{019M}~and~\href{https://stacks.math.columbia.edu/tag/019S}{019S}]{Stacks}.
By Lemma \ref{Lem componentwise acyclic}, the termwise split short
exact sequence \eqref{Dia termwise split} gives the desired
distinguished triangle. The obtained sequence does not depend on the
choice of $\mathbf I$ since any two injective resolutions are
(non-canonically) homotopy equivalent \cite[Tag
\href{https://stacks.math.columbia.edu/tag/05TG}{05TG}]{Stacks}.
%
For functoriality, let
\begin{equation*}
\begin{tikzcd}
(Y'_\bullet,L') \ar{d}\ar{r}{(f'_\bullet,\phi')} & (X'_\bullet,K') \ar{d} \\
(Y_\bullet,L) \ar{r}{(f_\bullet,\phi)} & (X_\bullet,K)
\end{tikzcd}
\end{equation*}
be a commutative diagram of simplicial objects in $\COMP$ where $K$,
$K'$, $L$, and $L'$ are bounded below.
%
We may view this as a bounded below complex $\mathbf K$ on the
$\Big(\begin{tikzpicture}[baseline=(a).base,inner sep=0pt]\node (a)
at (0,-.08){};\node[scale=0.7] at (0,0) {\begin{tikzcd}[row
sep=.3em, column sep=.3em,every arrow/.append
style={shorten=-.25em}]\cdot & \cdot \ar{l} \\ \cdot \ar{u} & \cdot
\ar{u}\ar{l}\end{tikzcd}};\end{tikzpicture}\Big)$-topos
\begin{equation*}
\begin{tikzcd}[column sep=1.3em,row sep=1.4em]
\mathbf \Gamma(Y') \ar{d}\ar{r} & \mathbf \Gamma(X') \ar{d} \\
\mathbf \Gamma(Y) \ar{r} & \mathbf \Gamma(X)\punct{.}
\end{tikzcd}
\end{equation*}
Proceeding as above gives a natural morphism from the sequence
\eqref{Dia termwise split} for $f_\bullet$ to the same sequence for
$f'_\bullet$ (which again does not depend on any choices).
\end{proof}

\section{All-or-nothing motives and additional axioms}\label{Sec all-or-nothing}

As in Section \ref{Sec motives} we will assume $k$ is an
algebraically closed field. We want to apply the previous section in
the case where $X_\bullet$ is a simplicial scheme, viewed as a
simplicial topos with the pro-\'etale topology \cite{BS} (resp.~the
crystalline topology \cite{BerCris}), and $K$ and $L$ are the
constant sheaves $\Q_\ell$ (resp.~the structure sheaf $\mathcal
O_{X/W(k)}$). The recipe of Definition \ref{Def simplicial cone} and
Remark \ref{Rmk terminal site} tells us to consider the sheaf on
$C_\bullet(f)$ given by $0$ on the components $*$ of $C_\bullet(f)$
and by $\Q_\ell$ (resp.~$\mathcal O_{-/W(k)}$) on all other
components.

Thus, we need to equip the mapping cone of a morphism of simplicial
schemes with a mild motivic structure, interpreting the sheaf $0$ on
$*$ as a variant of the zero motive. This motivates the following ad
hoc notion.

\begin{Def}
An \emph{all-or-nothing motive} $(X,p)$ is a smooth projective
$k$-scheme together with a locally constant function $p \colon X \to
\{0,1\}$. The set of morphisms $f \colon (X,p) \to (Y,q)$ of
all-or-nothing motives $(X,p)$, $(Y,q)$ whose underlying schemes $X$
and $Y$ are connected is given by
\[
\Mor\big((X,p),(Y,q)\big) = \left\{\begin{array}{ll}\Mor(Y,X), &
p=q=1, \\ \{0\}, & p=0, \\ \varnothing, & \text{else,} \end{array}
\right.
\]
with composition given by $f \circ 0 = 0$. (The contravariance is
consistent with Remark \ref{Rmk motives contravariant}.)

In general, if $X = \coprod X_i$ and $Y = \coprod Y_j$ with the
$X_i$ and $Y_j$ irreducible, we set
\[
\Mor\big((X,p),(Y,q)\big) = \prod_j \coprod_i
\Mor\left(\big(X_i,p\big|_{X_i}\big),\big(Y_j,
q\big|_{Y_j}\big)\right).
\]
The category of all-or-nothing motives is denoted $\M_k\aon$.
\end{Def}

\begin{Rmk}\label{Rmk all-or-nothing motives as motives}
If $X_i$ are the components of $X$, then we think of $(X,p)$ as the
pure motive $(X,p,0)$ by identifying $p$ with the projector in
$\Corr(X,X)$ given by $0 \in \CH^*(X_i \times X_i)$ if $p|_{X_i} =
0$ and by $\Delta_{X_i} \in \CH^*(X_i \times X_i)$ if $p|_{X_i} =
1$. This gives a factorisation $\SmPr_k\op \to \M_k\aon \to \M_k$ of
the functor $\SmPr_k\op \to \M_k$. If $H$ is a Weil cohomology
theory, then the extension $H \colon \M_k \to \gVec_K$ of Definition
\ref{Def pushforward on Weil cohomology} gives functors $H \colon
\M_k\aon \to \gVec_K$ as well.

Note that the all-or-nothing motives $(X,0)$ for $X$ irreducible are
all isomorphic: the maps $0 \colon (X,0) \to (Y,0)$ and vice versa
are mutual inverses. This gives an alternative construction for
$\M_k\aon$:
\end{Rmk}

\begin{Def}
Let $\mathscr C$ be a category. Define the category $\mathscr C
\amalg \{*\}$ whose objects are $\operatorname{ob} \mathscr C \amalg
\{*\}$, and morphisms are given by
\[
\Mor_{\mathscr C \amalg\{*\}}(X,Y) =
\left\{\begin{array}{ll}\Mor_{\mathscr C}(X,Y), & X,Y \in
\operatorname{ob} \mathscr C, \\ \{0\}, & Y = *, \\ \varnothing, &
\text{else,}\end{array}\right.
\]
where $0 \circ f = 0$ whenever this makes sense.
If $F \colon \mathscr C \to \mathscr D$ is a functor to a category
$\mathscr D$ with a terminal object $*$, then there is a unique
extension of $F$ to a functor $F \colon \mathscr C \amalg \{*\} \to
\mathscr D$ with $F(*) = *$.
\end{Def}

\begin{Def}
Let $\mathscr C$ be a category. Define the category
$\operatorname{Coprod}(\mathscr C)$ of \emph{formal finite
coproducts in $\mathscr C$} whose objects are diagrams $X \colon I
\to \mathscr C$ from a finite discrete category $I$, and morphisms
from $X \colon I \to \mathscr C$ to $Y \colon J \to \mathscr C$ are
given by
\[
\Mor(X,Y) = \prod_{i \in I} \coprod_{j \in J} \Mor(X_i,Y_j).
\]
If $F \colon \mathscr C \to \mathscr D$ is a functor to a category
$\mathscr D$ with finite coproducts, then there is a unique
extension (up to isomorphism) of $F$ to a functor $F \colon
\operatorname{Coprod}(\mathscr C) \to \mathscr D$ taking $X \colon I
\to \mathscr C$ to the coproduct $\coprod_i F(X_i)$.

If $X \colon I \to \mathscr C$ is an object of
$\operatorname{Coprod}(\mathscr C)$, then $X$ is the coproduct of
the one-object diagrams $X_i$, so we may write $X = \coprod_i X_i$.
An initial object of $\operatorname{Coprod}(\mathscr C)$ is given by
the empty diagram $X \colon \varnothing \to \mathscr C$, and if $*$
is a terminal object of $\mathscr C$, then it is also terminal in
$\operatorname{Coprod}(\mathscr C)$.
\end{Def}

\begin{Ex}\label{Ex Coprod}
We have an equivalence $\SmPr \cong
\operatorname{Coprod}(\SmPrVar_k)$, and $\M_k\aon$ can be defined as
$(\operatorname{Coprod}(\SmPrVar_k \amalg \{*\}))\op$. In
particular, any functor $F \colon \SmPrVar_k \to \mathscr D$ to a
category $\mathscr D$ with finite coproducts and a terminal object
$*$ extends uniquely to a functor $F \colon (\M_k\aon)\op \to
\mathscr D$ such that $F(X,0) = *$ for $X$ irreducible and $F(X,p) =
\coprod_i F(X_i, p|_{X_i})$ if the $X_i$ are the components of $X$.
This gives an alternative method to extend a Weil cohomology $H$ to
a functor $H \colon \M_k\aon \to \gVec_K$, cf.~Remark \ref{Rmk
all-or-nothing motives as motives}.
\end{Ex}

Now we are ready to state the additional axioms on our Weil
cohomology theory.

\begin{Ax}\label{Def additional axioms}
Consider the following axioms on a Weil cohomology theory $H$.
\begin{enumerate}
\fancyitem{\rm (A\arabic{enumi})} \label{Ax compactly supported
cohomology} The functor $H \colon \SmPr_k\op \to \gAlg_K$ extends to
a compactly supported cohomology functor $H\cs \colon \mathscr C\op
\to \gVec_K$ where $\mathscr C \subseteq \Sch_{\text{sep, f.t.}/k}$
is the subcategory of proper morphisms. For a closed immersion $Z
\hookrightarrow X$ in $\mathscr C$ with complementary open $U$, the
pullback maps $H^i\cs(X) \to H^i\cs(Z)$ fit into a long exact
sequence
\[
\ldots \to H^i\cs(U) \to H^i\cs(X) \to H^i\cs(Z) \to \ldots
\]
functorial for proper morphisms $f \colon Y \to X$ by pulling back
$Z$ and $U$. If $X$ is proper, we write $H^i(X)$ for $H^i\cs(X)$.
\fancyitem{\rm (A\arabic{enumi})}\label{Ax vanishing} For any $X \in
\mathscr C$, we have $H^i\cs(X) = 0$ for $i > 2 \dim X$. If $X$ is
smooth projective of dimension $n$ and $Z \subseteq X$ is a reduced
closed subscheme whose irreducible components $Z_1,\ldots,Z_r$ have
dimension $d$, then $H^{2d}(Z) \cong K^r$ and $H^{2d}(X) \to
H^{2d}(Z) \cong K^r$ is Poincar\'e dual to the map
\begin{align*}
K^r &\to H^{2(n-d)}(X)\\
(\lambda_1,\ldots,\lambda_r) &\mapsto \sum_{i=1}^r \lambda_i
\cl(Z_i).
\end{align*}
\fancyitem{\rm (A\arabic{enumi})}\label{Ax cohomology of simplicial
motive} For a cosimplicial all-or-nothing motive
$(X_\bullet,\pi_\bullet)$, there is a graded $K$-vector space
$H^*(X_\bullet,\pi_\bullet)$ that is computed by a spectral sequence
\[
E_1^{p,q} = H^q(X_p,\pi_p) \Ra H^{p+q}(X_\bullet,\pi_\bullet)
\]
that is functorial in $(X_\bullet,\pi_\bullet)$, where $E_1^{*,q}$
is the alternating face complex on the cosimplicial $K$-vector space
$H^q(X_\bullet,\pi_\bullet)$. If all $\pi_p$ are equal to the
constant function $1$, then we write $H^*(X_\bullet)$ for
$H^*(X_\bullet,\pi_\bullet)$. \fancyitem{\rm
(A\arabic{enumi})}\label{Ax cohomological descent} (Cohomological
descent for proper hypercoverings) If $X_\bullet \to X$ is a proper
hypercovering of a proper $k$-scheme $X$ such that all $X_n$ are
smooth projective $k$-schemes, then the pullback map
\[
H^*(X) \to H^*(X_\bullet)
\]
is an isomorphism. \fancyitem{\rm (A\arabic{enumi})}\label{Ax
mapping cone} If $f_\bullet \colon (X_\bullet,\alpha_\bullet) \to
(Y_\bullet,\beta_\bullet)$ is a morphism of cosimplicial
all-or-nothing motives, then there is a long exact sequence
\[
\ldots \to H^i(C_\bullet(f)) \to H^i(X_\bullet,\beta_\bullet)
\stackrel{f_{\bullet,*}}\longrightarrow
H^i(Y_\bullet,\alpha_\bullet) \to \ldots
\]
that is functorial in the morphism $f_\bullet$, where $C_\bullet(f)$
is the mapping cone of cosimplicial all-or-nothing motives,
cf.~Definition \ref{Def simplicial cone}. \fancyitem{\rm
(A\arabic{enumi})}\label{Ax comparison of long exact sequence} If $i
\colon Z \hookrightarrow X$ is a closed immersion of proper
$k$-schemes with complement~$U$, if $X_\bullet \to X$ and $Z_\bullet
\to Z$ are proper hypercoverings such that all $X_n$ and $Z_n$ are
smooth projective $k$-schemes, and if $i_\bullet \colon Z_\bullet
\to X_\bullet$ is a morphism of simplicial schemes fitting in a
commutative diagram
\begin{equation*}
\begin{tikzcd}
Z_\bullet \ar{r}{i_\bullet}\ar{d} & X_\bullet \ar{d} & \\
Z \ar[closed]{r}{i} & X & U \ar[open']{l}\punct{,}
\end{tikzcd}
\end{equation*}
then the isomorphisms $H^*(X) \stackrel\sim\to H^*(X_\bullet)$ and
$H^*(Z) \stackrel\sim\to H^*(Z_\bullet)$ of \ref{Ax cohomological
descent} and the long exact sequences of \ref{Ax compactly supported
cohomology} and \ref{Ax mapping cone} give a commutative diagram
\begin{equation}\label{Dia comparison}
\begin{tikzcd}
\ldots \ar{r} & H^i\cs(U) \ar{r}\ar{d}[rotate=90,anchor=north,xshift=.1em]{\sim} & H^i(X) \ar{r}\ar{d}[rotate=90,anchor=north,xshift=.1em]{\sim} & H^i(Z) \ar{r}\ar{d}[rotate=90,anchor=north,xshift=.1em]{\sim} & \ldots\\
\ldots \ar{r} & H^i(C_\bullet(i)) \ar{r} & H^i(X_\bullet) \ar{r} &
H^i(Z_\bullet) \ar{r} & \ldots\punct{,}
\end{tikzcd}
\end{equation}
functorial for commutative diagrams
\begin{equation}\label{Dia square of hypercoverings}
\begin{tikzcd}[row sep=.6em, column sep=.3em,crossing over clearance=1ex]
 & V_\bullet \ar{rrr}\ar{dd} & \makebox*{.}{} & & X_\bullet \ar{dd} & \makebox*{}{} & & \\
W_\bullet \ar[start anchor=center, end anchor=center, shorten=1em]{ru}\ar[rrr,crossing over]\ar{dd} & & & Y_\bullet \ar[start anchor=center, end anchor=center, shorten=1em]{ru} & & & & \\
 & V \ar[closed]{rrr} & & & X & & & \makebox*{$X$}[l]{$X \setminus V$}\ar[open']{lll} \\
W \ar[start anchor=center, end anchor=center,
shorten=1em]{ru}\ar[closed]{rrr} & & & Y \ar[start anchor=center,
end anchor=center, shorten=1em]{ru}\ar[from=uu, crossing over] & & &
\makebox*{$Y$}[l]{$Y \setminus W$} \ar[open']{lll}\ar[start
anchor=center, end anchor=center, shorten=1em,xshift=.4em]{ru} &
\end{tikzcd}
\end{equation}
where the bottom squares (but not necessarily the top square) are
pullbacks and all vertical maps are proper hypercoverings by smooth
projective $k$-schemes.
\end{enumerate}
\end{Ax}

\begin{Ex}\label{Ex etale additional}
If $\ell$ is a prime invertible in $k$, then $\ell$-adic \'etale
cohomology is a Weil cohomology theory by Example \ref{Ex etale}. It
satisfies additional axiom \ref{Ax compactly supported cohomology}
by \cite[exp.~XVII,~5.1.16]{SGA4III}. The vanishing statement in
\ref{Ax vanishing} is given by
\cite[exp.~XVII,~cor.~5.2.8.1]{SGA4III}. For the computation of
$H^{2d}(Z)$ we may therefore pass to an open and assume $Z$ is
smooth, where the result follows from Poincar\'e duality
\cite[exp.~XVIII,~th.~3.2.5]{SGA4III}. The final statement of
\ref{Ax vanishing} follows from the definition of the cycle class
map \cite[chap.~4,~d\'ef.~2.3.2]{SGA4.5}.

By a (pseudofunctor version of) the universal property of Example
\ref{Ex Coprod}, the pseudofunctor
\begin{align*}
(-,\Q_\ell) \colon \SmPrVar &\to \mathscr \SHV\\
X &\mapsto (X\proet, \Q_\ell)
\end{align*}
extends uniquely to a pseudofunctor $F \colon (\M_k\aon)\op \to
\mathscr \SHV$ that maps $*$ to the pair $(\Set,0)$ and preserves
finite coproducts (see Remark \ref{Rmk terminal site}). This
corresponds to a fibred and cofibred category $E \to \M_k\aon$ whose
fibre $(X,p)\proet$ over a connected all-or-nothing motive $(X,p)$
is $X\proet$ if $p = 1$ and $\Set$ if $p = 0$, along with a sheaf
$\mathscr F$ on $\mathbf \Gamma(E)$ whose restriction to $(X,p)$ is
$\Q_\ell$ if $p = 1$ and $0$ if $p = 0$.

Given a cosimplicial all-or-nothing motive $(X_\bullet,\pi_\bullet)
\colon \Delta \to \M_k\aon$, the fibre product $E \times_{\M_k\aon}
\Delta \to \Delta$ is a simplicial topos
\cite[exp.~V$^{\text{bis}}$,~1.2.5]{SGA4II}, which we will denote by
$(X_\bullet,\pi_\bullet)\proet$.

The sheaf $\mathscr F$ on $\mathbf \Gamma(E)$ pulls back to $\mathbf
\Gamma((X_\bullet,\pi_\bullet)\proet)$, and the spectral sequence of
Remark \ref{Rmk simplicial notation} then reads
\begin{equation*}
E_1^{p,q} = H^q\Big((X_p,\pi_p)\proet,\mathscr F\Big) \Ra
H^{p+q}\big((X_\bullet,\pi_\bullet),\mathscr F\big).
\end{equation*}
For any all-or-nothing motive $(X,\pi)$, we have
$H^q((X,\pi)\proet,\mathscr F) = H^q(X,\pi)$, where the right hand
side is defined by Remark \ref{Rmk all-or-nothing motives as
motives}. This gives the required spectral sequence of \ref{Ax
cohomology of simplicial motive}. Moreover, \ref{Ax mapping cone}
holds by Lemma \ref{Lem distinguished triangle}, since all
(pseudo)functors involved preserve terminal objects and finite
coproducts, hence preserve the construction of the mapping cone.
Cohomological descent \ref{Ax cohomological descent} follows from
\cite[exp.~V$^{\text{bis}}$, prop.~4.3.2, th.~3.3.3,
and~prop.~2.5.7]{SGA4II}.

Finally, in the situation of \ref{Ax comparison of long exact
sequence}, the functoriality statement in Lemma \ref{Lem
distinguished triangle} immediately reduces us to the case of the
trivial hypercoverings $X_\bullet \to X$ and $Z_\bullet \to Z$ given
by the constant simplicial schemes $X_n = X$ and $Z_n = Z$. (This
does not preserve the hypothesis that each $X_n$ and $Z_n$ is smooth
projective, but \'etale cohomology is defined and has all the
desired properties also for simplicial objects in $\Sch_{\text{sep,
f.t.}/k}$.)

In fact, since $i_*$ is exact, the same goes for $\mathbf
\Gamma(i)_* \colon \Ab(Z_\bullet) \to \Ab(X_\bullet)$, so we may
replace the morphism $(Z_\bullet,\Q_{\ell,\bullet}) \to
(X_\bullet,\Q_{\ell,\bullet})$ by $(X_\bullet,i_*\Q_{\ell,\bullet})
\to (X_\bullet,\Q_{\ell,\bullet})$ using the functoriality assertion
of Lemma \ref{Lem distinguished triangle}. An injective resolution
\begin{equation*}
\begin{tikzcd}[row sep=1.5em,column sep=1.5em]
\Q_\ell \ar{r}{\alpha}\ar{d} & i_*\Q_\ell \ar{d} \\
I^* \ar{r}{\beta} & J^*
\end{tikzcd}
\end{equation*}
on $X\proet$ gives a constant (degreewise injective) resolution on
$(X_\bullet)\proet$
\begin{equation}\label{Dia injective resolution}
\begin{tikzcd}[row sep=1.5em,column sep=1.5em]
\Q_{\ell,\bullet} \ar{r}{\alpha_\bullet}\ar{d} & i_*\Q_{\ell,\bullet} \ar{d} \\
I^*_\bullet \ar{r}{\beta_\bullet} & J^*_\bullet\punct{.}
\end{tikzcd}
\end{equation}
This gives a commutative diagram with exact rows
\begin{equation}\label{Dia cone injective resolution}
\begin{tikzcd}[row sep=1.5em,column sep=1.5em]
0 \ar{r} & \makebox*{$\Tot\big(C^*(\beta_\bullet)\big)$}{$j_!\Q_\ell$} \ar{d}\ar{r} & \makebox*{$\Tot\big(\!\cyl^*(\beta_\bullet)\big)$}{$\Q_\ell$} \ar{d}\ar{r} & \makebox*{$\Tot\big(i_*\Q_{\ell,\bullet}\big)$}{$i_*\Q_\ell$} \ar{d}\ar{r} & 0\\
0 \ar{r} & \Tot\big(C^*(\alpha_\bullet)\big) \ar{d}\ar{r} & \Tot\big(\!\cyl^*(\alpha_\bullet)\big) \ar{d}\ar{r} & \Tot\big(i_*\Q_{\ell,\bullet}\big) \ar{d}\ar{r} & 0\\
0 \ar{r} & \Tot\big(C^*(\beta_\bullet)\big) \ar{r} &
\Tot\big(\!\cyl^*(\beta_\bullet)\big) \ar{r} &
\makebox*{$\Tot\big(i_*\Q_{\ell,\bullet}\big)$}{$\Tot\big(J^*_\bullet\big)$}
\ar{r} & 0
\end{tikzcd}
\end{equation}
of complexes on $X\proet$, whose vertical maps are
quasi-isomorphisms, where $C^*$ and $\cyl^*$ denote the dual
constructions to Definition \ref{Def simplicial cone}, applied to
the rows of \eqref{Dia injective resolution} viewed as morphisms of
cosimplicial objects in $\Ch(X\proet)$. The terms in the bottom row
of \eqref{Dia cone injective resolution} are injective, so the long
exact sequence of \ref{Ax compactly supported cohomology} is
computed by
\[
0 \to \Tot\Big(\Gamma\big(X,C^*(\beta_\bullet)\big)\Big) \to
\Tot\Big(\Gamma\big(X,\cyl^*(\beta_\bullet)\big)\Big) \to
\Tot\Big(\Gamma\big(X,J^*_\bullet\big)\Big) \to 0.
\]
This agrees with the sequence \eqref{Dia termwise split} for the
constant resolution \eqref{Dia injective resolution}, giving the
commutative diagram \eqref{Dia comparison}.
%
This construction is functorial for commutative diagrams \eqref{Dia
square of hypercoverings} since pullback along $Y \to X$ preserves
the short exact sequence $0 \to j_!\Q_\ell \to \Q_\ell \to
i_*\Q_\ell \to 0$ and by the functoriality statement in Lemma
\ref{Lem distinguished triangle}.
\end{Ex}

\begin{Ex}\label{Ex crystalline additional}
If $k$ is a perfect field of positive characteristic $p$ with Witt
ring $W(k)$ with field of fractions $K$, then crystalline cohomology
$H^i_{\text{cris}}(X/K)$ is a Weil cohomology theory by Example
\ref{Ex crystalline}. Axiom \ref{Ax compactly supported cohomology}
is provided by rigid cohomology with compact support
\cite[3.1(iii)]{BerRig}, and \ref{Ax cohomological descent} is
\cite[Theorem~4.5.1]{Tsu}. The proofs of axioms \ref{Ax vanishing}
(using Poincar\'e duality for rigid cohomology of smooth varieties
\cite{BerDual}), \ref{Ax cohomology of simplicial motive}, \ref{Ax
mapping cone}, and \ref{Ax comparison of long exact sequence} (using
Le Stum's site theoretic definition of rigid cohomology
\cite{LeStum}) are analogous to Example \ref{Ex etale additional}.
\end{Ex}

\begin{Rmk}\label{Rmk Cisinski-Deglise}
The above axioms are exactly what we need, but possibly not the most
natural choice. It's likely that our axioms (or an alternative set
of sufficient axioms) can be deduced from Cisinski--D\'eglise's
axioms for a \emph{mixed Weil cohomology theory} \cite{CisDeg1}.
Somewhat surprisingly, cohomological descent for proper
hypercoverings \ref{Ax cohomological descent} is indeed always
satisfied \cite[Corollary~17.2.6]{CisDeg2}.
\end{Rmk}

\section{Independence of Weil cohomology theory}\label{Sec independence}
As in Section \ref{Sec motives}, we assume that $k$ is an
algebraically closed field, and we fix an adequate equivalence
relation $\sim$ that is finer than homological equivalence for any
Weil cohomology theory.
%
From now on, we will fix Weil cohomology theories $H$ and $\H$ (see
Definition \ref{Def Weil cohomology theory}) with coefficient fields
$K$ and $\K$ respectively. We will always assume that $H$ and $\H$
satisfy the additional properties \hyperref[Ax compactly supported
cohomology]{(A1--6)} of Axiom \ref{Def additional axioms}.

\begin{Def}\label{Def conjectures}
Let $X$ and $Y$ be separated $k$-schemes of finite type. Consider
the following statements on independence of Weil cohomology theory:
{\setitemize{leftmargin=*, align=parleft, labelsep=2cm}
\begin{itemize}
\item[$\Dim\cs(X)$:] for each $i$, the dimensions of $H^i\cs(X)$ and $\H^i\cs(X)$ agree.
\item[$\Rk\cs(X,Y)$:] for any proper morphism $f \colon Y \to X$ and any $i$, the ranks of $f^*$ on $H^i\cs$ and $\H^i\cs$ agree.
\end{itemize}

\smallskip\noindent
If $X$ and $Y$ are smooth projective, we further consider:
\smallskip

\begin{itemize}
\item[$\Cl(X)$:] the kernels of the cycle class maps $\cl \colon \CH^*_\Q(X) \to H^*(X)$ and $\cl \colon \CH^*_\Q(X) \to \H^*(X)$ agree.
\item[$\Rk^r(X,Y)$:] for any $\alpha \in \Corr^r(X,Y)$ and any $i$, the ranks of $\alpha_*$ on $H^i$ and $\H^i$ agree.
\item[$\Rk(X,Y)$:] $\Rk^0(X,Y)$ holds.
\item[$\Rk^*(X,Y)$:] $\Rk^r(X,Y)$ holds for all $r \in \Z$.
\item[$\Kun(X)$:] for each $i$, there exists a cycle $p \in \Corr(X,X)$ inducing the $i^{\text{th}}$ K\"unneth projector on both $H^*$ and $\H^*$.
\end{itemize}}

\smallskip
\noindent The reliance of these properties on the chosen Weil
cohomology theories $H$ and $\H$ will be implicit, and we will make
no further mention of it.
\end{Def}

\begin{Rmk}\label{Rmk Rk motives}
For Chow motives $M = (X,p,m)$, $N = (Y,q,n)$ one can also define similar statements $\Dim(M)$ and $\Rk(M,N)$. But these are already
implied by $\Rk(X,X)$ and $\Rk^*(X,Y)$ respectively: the dimension
of $H^i(M)$ is the rank of $p_* \colon H^i(X) \to H^i(X)$, and the
rank of $\alpha_* \colon H^i(M) \to H^i(N)$ is the rank of $(q\alpha
p)_* \colon H^{i+2m}(X) \to H^{i+2n}(Y)$.
\end{Rmk}

\begin{Rmk}\label{Rmk Kunneth projectors}
For $\Kun(X)$, note that such a cycle $p$ need not be a projector in
$\Corr(X,X)$. We only know that the cycle class map sends it to a
projector in both $\End(H^*(X))$ and $\End(\H^*(X))$.
\end{Rmk}

\begin{Rmk}\label{Rmk finite field}
If $k = \bar \F_p$ and $X$ is smooth proper, then we know
$\Dim\cs(X)$ (for \emph{any} Weil cohomology theories $H$, $\H$),
because the dimension can be read off from the zeta function. On the
other hand, $\Rk\cs(X,Y)$ is still unknown even when $X$ and $Y$ are
smooth and projective.

We also know $\Kun(X)$ for $X$ smooth projective over $\bar \F_p$,
by \cite{KM}. Hence by Corollary \ref{Cor characteristic
polynomial}, the characteristic polynomial of $\alpha \in
\Corr(X,X)$ is independent of the Weil cohomology theory. In
particular, if $\alpha = p$ is a projector, this implies that $\dim
H^i(X,p,0)$ is independent of $H$.
\end{Rmk}

\begin{Rmk}
If $\operatorname{char}k = 0$, then for all known cohomology
theories $H$ and $\H$, the statements $\Cl(X)$ and $\Rk^*(X,Y)$ for
$X$ and $Y$ smooth projective, as well as $\Dim\cs(X)$ and
$\Rk\cs(X,Y)$ for $X$ and $Y$ separated and of finite type over $k$
are known. On the other hand, $\Kun(X)$ is still open, even for
(smooth projective) varieties over $\C$.
\end{Rmk}


\begin{Thm}\label{Thm equivalent}
Let $k = \bar\F_p$. Then the following are equivalent:
\begin{enumerate}
\fancyitem{(\arabic{enumi})}\label{Item 1} For all smooth projective
$k$-schemes $X$, we have $\Cl(X)$.
\initiate{enumi} \fancyitem{(2\alph{enumi})}\label{Item 2a} For all
smooth projective $k$-schemes $X$ and $Y$, we have $\Rk^*(X,Y)$;
\fancyitem{(2\alph{enumi})}\label{Item 2b} For all smooth projective
$k$-schemes $X$ and $Y$, we have $\Rk(X,Y)$;
\fancyitem{(2\alph{enumi})}\label{Item 2c} For all smooth projective
$k$-schemes $X$, we have $\Rk(X,X)$;
\initiate{enumi} \fancyitem{(3\alph{enumi})}\label{Item 3a} For all
separated finite type $k$-schemes $X$ and $Y$, we have $\Dim\cs(X)$
and $\Rk\cs(X,Y)$; \fancyitem{(3\alph{enumi})}\label{Item 3b} For
all smooth, quasi-projective $k$-schemes $X$, we have $\Dim\cs(X)$.
\end{enumerate}
\end{Thm}

The outline of the rest of the article is as follows. In each of the
following sections, we will prove one of the implications, often in
a more refined version. We will prove the implications in the
following cyclic order:
\begin{equation*}
\begin{tikzcd}
\ref{Item 1} \ar[Leftrightarrow]{r} & \ref{Item 2a} \ar[Rightarrow]{r} & \ref{Item 2b} \ar[Rightarrow]{r}\ar[Leftrightarrow]{d} & \ref{Item 3a} \ar[Rightarrow]{r} & \ref{Item 3b} \ar[Rightarrow]{r} & \ref{Item 1}\punct{.} \\
& & \ref{Item 2c} & & &
\end{tikzcd}
\end{equation*}
Implications $\ref{Item 2a} \Ra \ref{Item 2b} \Ra \ref{Item 2c}$ and
$\ref{Item 3a} \Ra \ref{Item 3b}$ are trivial. For $\ref{Item 2c}
\Ra \ref{Item 2b}$, recall that
\[
\Corr(X,Y) = \Hom_{\M_k}\big((X,\id,0),(Y,\id,0)\big),
\]
so $\Rk(X \amalg Y, X \amalg Y)$ implies $\Rk(X, Y)$ since $X \amalg
Y$ is the biproduct in $\M_k$ (see Remark \ref{Rmk biproducts}), the
functors $H \colon \M_k \to \gVec_K$ and $\H \colon \M_k \to
\gVec_\K$ preserve biproducts, and the rank of a block matrix
$\left(\begin{smallmatrix}0 & 0 \\ A & 0\end{smallmatrix}\right)$ is
the rank of $A$.

The implications $\ref{Item 1} \LRa \ref{Item 2a}$, $\ref{Item 2b}
\Ra \ref{Item 3a}$, and $\ref{Item 3b} \Ra \ref{Item 1}$ will be the
contents of the following three sections (Section \ref{Sec Cycle
classes and ranks}, Section \ref{Sec Ranks and dimensions}, Section
\ref{Sec Dimensions and cycle class maps}) respectively.

\begin{Rmk}
The proof of $\ref{Item 1} \Ra \ref{Item 2a}$ relies on $\Kun(X)$,
which currently is known only when $k = \bar \F_p$. Implication
$\ref{Item 2b} \Ra \ref{Item 3a}$ uses the Weil conjectures and a
hypercovering argument, so also does not generalise to other fields
(but see Remark \ref{Rmk other fields}). The implication $\ref{Item
3b} \Ra \ref{Item 1}$ holds over an arbitrary algebraically closed
field.
\end{Rmk}

\begin{Rmk}\label{Rmk other fields}
Suppose $H$ and $\H$ are given by \'etale cohomology. If $S$ is an
irreducible scheme, $\bar s$ and $\bar \eta$ are a geometric point
and a geometric generic point, $\ell$ is a prime invertible on $S$,
and $X \to S$ is a smooth proper morphism of schemes, then the
smooth and proper base change theorems give isomorphisms
\begin{equation}\label{Eq specialisation}
\operatorname{sp} \colon H^i\big(X_{\bar s},\Q_\ell\big) \cong
H^i\big(X_{\bar \eta},\Q_\ell\big);
\end{equation}
see for example \cite[exp.~XVI,~cor.~2.2]{SGA4III}. Using a standard
spreading out argument, this shows that statements \ref{Item 1} and
\hyperref[Item 2a]{(2abc)} for $\bar \F_p$ for one prime $p$ (resp.\
every prime $p$) imply the same result for any algebraically closed
field $k$ of characteristic $p$ (resp.\ any algebraically closed
field $k$).

Moreover, for \'etale cohomology the argument in $\ref{Item 2b} \Ra
\ref{Item 3a}$ can be refined to deduce \ref{Item 3a} over an
arbitrary algebraically closed field $k$ from \ref{Item 2b} over
$\bar \F_p$, again using spreading out and the specialisation
isomorphism \eqref{Eq specialisation}. The argument for $\ref{Item
3b} \Ra \ref{Item 1}$ works over any algebraically closed field, so
we see that the case $k = \bar \F_p$ is the essential one.
\end{Rmk}

\begin{Rmk}
If $X$ is smooth projective over $k$ and $\sim$ is rational
equivalence, we get a ring isomorphism
\begin{equation}
\operatorname{ch} \colon K_\Q(X) \to \CH^*_\Q(X).\label{Eq GRR}
\end{equation}
Thus, $\Cl(X)$ is equivalent to the following statement:

\smallskip
{\setitemize{leftmargin=*, align=parleft, labelsep=2cm}
\begin{itemize}
\item[$\Cl'(X)$:] the kernels of the Chern character maps $\operatorname{ch} \colon K_\Q(X) \to H^*(X)$ and $\operatorname{ch} \colon K_\Q(X) \to \H^*(X)$ agree.
\end{itemize}}

\smallskip
\noindent To study the vanishing of $\operatorname{ch}_H(\alpha)$
for $\alpha \in K_\Q(X)$, the splitting principle plus injectivity
of pullbacks for dominant maps \cite[Proposition~1.2.4]{KleDix} reduces us
to the case where $\alpha$ is in the subring of $K_\Q(X)$ generated
by classes of the form $[\mathscr L]$ for $\mathscr L$ a line bundle
on $X$. Under the isomorphism \eqref{Eq GRR}, this corresponds to
the subalgebra of $\CH^*_\Q(X)$ generated by divisors (but note that
$\operatorname{ch}([D]) \neq [\mathcal O(\pm D)]$).

Although $\Cl(X)$ is known for divisors, it seems that this cannot
be used to deduce the statement in general.
\end{Rmk}

\section{Cycle classes and ranks}\label{Sec Cycle classes and ranks}
In this section, $k$ is an arbitrary algebraically closed field.
However, we will soon assume that $\Kun$ holds, which is currently
only known for $k = \bar \F_p$ \cite{KM}.

\begin{Thm} Let $X$ and $Y$ be smooth projective $k$-schemes.
\begin{enumerate}
\fancyitem{(\arabic{enumi})}\label{Item Thm Cl 1} Assume
$\Rk^*(\Spec k, X)$. Then $\Cl(X)$ holds.
\fancyitem{(\arabic{enumi})}\label{Item Thm Cl 2} Assume $\Kun(X)$,
$\Kun(Y)$, and $\Cl((X\times Y)^n)$ for all $n$. Then $\Rk^*(X,Y)$
holds.
\end{enumerate}
\end{Thm}

\begin{proof}
We have $\Corr^r(\Spec k,X) = \CH^r_\Q(X)$, and a cycle $\alpha \in
\CH^r_\Q(X)$ maps to zero under the cycle class map $\cl \colon \CH^r_\Q(X) \to
H^{2r}(X)$ if and only if $\alpha_* \colon H^*(\Spec k) \to H^*(X)$
is zero (and similarly for $\H^*$); see Lemma \ref{Lem vanish}. Now
\ref{Item Thm Cl 1} follows from the assumption that $H^*(\Spec k)$
is $1$-dimensional, so the only possibilities for the rank of
$\alpha_*$ are $0$ and $1$, corresponding to $\alpha_* = 0$ and
$\alpha_* \neq 0$ respectively.

For \ref{Item Thm Cl 2}, let $i$ and $r$ be given, and let $p \in
\Corr(X,X)$ (resp.~$q \in \Corr(Y,Y)$) be an element acting on $H^*$
and $\H^*$ as the $i^{\text{th}}$ (resp.~$i+2r^{\text{th}}$)
K\"unneth projector. For $\alpha \in \Corr^r(X,Y)$, we get an
induced element \vspace{-.2em}
\[
q\circ\alpha\circ p \in \Corr^r(X,Y).
\]
Moreover, the map $(q\alpha p)_* \colon H^i(X) \to H^{i+2r}(Y)$
agrees with the map induced by $\alpha$ (and the same holds for
$\H^i(X) \to \H^{i+2r}(Y)$). Denote this map by $\alpha_i$ (on both
$H^i$ and $\H^i$). First assume $i$ is even, and consider the
induced maps
\[
{\textstyle\bigwedge\nolimits}^j (q\alpha p) \colon
{\textstyle\bigwedge\nolimits}^j X \cto
\textstyle{\bigwedge\nolimits}^j Y
\]
for various $j$. 
%
By Remark \ref{Rmk super}, we have a decomposition
\[
H^*\Big({\textstyle\bigwedge}^j X\Big) =
\bigoplus_{a+b=j}{\textstyle\bigwedge}^a H\even(X) \otimes S^b
H\odd(X).
\]
The map $H^*(\bigwedge^j X) \to H^*(\bigwedge^j Y)$ induced by
$\bigwedge^j(q\alpha p)$ is $\bigwedge^j\alpha_i$ on $\bigwedge^j
H^i(X)$, and $0$ on all other components of $H^*(\bigwedge^j X)$. In
particular, it is nonzero if and only if $j \leq \rk(\alpha_i|_H)$.
Similarly, the map on $\H^*(X)$ induced by $\bigwedge^j(q\alpha p)$
is nonzero on $\H^*(X)$ if and only if $j \leq \rk(\alpha_i|_\H)$.
Thus, the rank of $\alpha_i$ only depends on the vanishing or
nonvanishing of the cycles $\bigwedge^j(q\alpha p)$ under the cycle
class map, by Lemma \ref{Lem vanish}. But we assumed that the
kernels of the cycle class maps are the same for $H^*((X\times
Y)^j)$ and $\H^*((X \times Y)^j)$. This proves the claim if $i$ is
even. If $i$ is odd, we use $S^j$ instead of $\bigwedge^j$ (see
Remark \ref{Rmk super}).
\end{proof}

\begin{Cor}
Let $k$ be an algebraically closed field such that $\Kun(X)$ holds
for all smooth projective $k$-schemes $X$ (e.g. $k = \bar\F_p$
\cite{KM}). Then the following are equivalent:
\begin{enumerate}
\fancyitem{(\arabic{enumi})} $\Cl(X)$ holds for all smooth
projective $k$-schemes $X$; \fancyitem{(\arabic{enumi})}
$\Rk^*(X,Y)$ holds for all smooth projective $k$-schemes $X$ and
$Y$. \hfill\qedsymbol
\end{enumerate}
\end{Cor}

\section{Ranks and dimensions}\label{Sec Ranks and dimensions}
In this section, $k$ will denote an arbitrary field. The main result
of this section (Theorem \ref{Thm dimension}) assumes that $k =
\bar\F_p$, because its proof relies on the Weil conjectures. The
idea is to use alterations to produce smooth hypercoverings that
compute the cohomology of arbitrary separated finite type
$k$-schemes.

\begin{Lemma}\label{Lem compactification}
Let $X$ and $Y$ be separated finite type $k$-schemes, and $f \colon
Y \to X$ a morphism. Then there exist proper $k$-schemes $\bar X$,
$\bar Y$ along with dense open immersions $X \to \bar X$ and $Y \to
\bar Y$ and a morphism $\bar f \colon \bar Y \to \bar X$ such that
the diagram
\begin{equation*}
\begin{tikzcd}
Y \ar{r}\ar{d}[swap]{f} & \bar Y \ar{d}{\bar f} \\
X \ar{r} & \bar X
\end{tikzcd}
\end{equation*}
commutes. If $f$ is proper, then $Y = \bar f^{-1}(X)$.
\end{Lemma}

\begin{proof}
Let $X \to \bar X$ be a Nagata compactification \cite{Nag1}.
Replacing $\bar X$ by the closure of $X$ in $\bar X$, we may assume
that $X$ is dense in $\bar X$. Let $\bar Y$ be a relative Nagata
compactification of $Y \to \bar X$ \cite{Nag2}. Again, we may assume
that $Y$ is dense in $\bar Y$.
Then $\bar Y$ is proper over $\bar X$, hence proper over $k$ since
$\bar X$ is. This proves the first statement. The second statement
follows because the scheme theoretic image of the morphism of proper
$X$-schemes $Y \to \bar f^{-1}(X)$ is closed. Since $Y$ is also
dense in $\bar f^{-1}(X)$ (in fact, in $\bar Y$), this forces
equality.
\end{proof}

\begin{Lemma}\label{Lem resolution of diagram}
Let $(\mathscr J, \leq)$ be a poset (viewed as category) such that
for every $i \in \mathscr J$ there are only finitely many $j \in
\mathscr J$ with $j \leq i$. Let $D_0 \colon \mathscr J\op \to
\Sch_k$ be a diagram of separated finite type $k$-schemes. Then
there exists a diagram $D_1 \colon \mathscr J\op \to \Sch_k$ and a
map $D_1 \to D_0$ such that
\begin{enumerate}
\fancyitem{(\arabic{enumi})} each $D_1(i) \to D_0(i)$ is proper and
surjective, and \fancyitem{(\arabic{enumi})} each $D_1(i)$ is
quasi-projective over $k$ and regular.
\end{enumerate}
\end{Lemma}

\begin{proof}
We will construct $D_1 \to D_0$ as a functor $D \colon (\mathscr J
\times [1])\op \to \Sch_k$ with the desired properties. For $i \in
\mathscr J$ and $J \subseteq \mathscr J$, write $J_{<i} = \{j \in J\
|\ j < i\}$ and $J_{\leq i} = \{j \in J\ |\ j \leq i\}$, and set
\[
C_{J,i} := \Big(\mathscr J_{\leq i} \times \{0\}\Big) \cup
\Big(J_{\leq i} \times \{1\}\Big) \subseteq \mathscr J_{\leq i}
\times [1].
\]
If $D$ is defined on $(J \times [1])\op$ and $i \in \mathscr J$ is
arbitrary, we write
\[
L_J(i) := \lim_{C_{J,i}\op}\ D.
\]
Since this is a finite limit, it exists in $\Sch_k$ and is separated
and of finite type. We always have $L_\varnothing = D_0$, and we
have $L_{\mathscr J_{\leq i}}(i) = D_1(i)$ and write $L(i) :=
L_{\mathscr J_{<i}}(i)$ when these are defined.

By induction on the size of the finite poset $\mathscr J_{<i}$, we construct $D_1(i)$
as a cone over the restriction of $D$ to $C_{\mathscr J_{<i},i} =
(\mathscr J_{\leq i} \times [1]) \setminus (i,1)$, such that
moreover the natural map $D_1(i) \to L(i)$ is proper and surjective.
If $i$ is minimal (i.e.\ $\mathscr J_{<i} = \varnothing$), take
$D_1(i) \twoheadrightarrow D_0(i)$ an alteration \cite[Theorem~4.1]{dJ}. Since $L(i) = L_\varnothing(i) = D_0(i)$, the map $D_1(i)
\to L(i)$ is proper and surjective by construction.

Now take $i \in \mathscr J$ arbitrary, and assume $D_1(j)$ has been
defined and $D_1(j) \to L(j)$ is proper and surjective for all $j$
with $|\mathscr J_{< j}| < |\mathscr J_{< i}|$; in particular for
all $j \in \mathscr J_{<i}$. Then $L(i)$ is defined, and we take
$D_1(i) \twoheadrightarrow L(i)$ any alteration \cite{dJ}. The
required functoriality $D_1(i) \to D_1(j)$ for $j < i$ comes from
the fact that $L(i)$ is a cone over the restriction of $D$ to
$C_{\mathscr J_{<i},i}$.

If $J \subseteq \mathscr J_{<i}$ is downward closed and $j \in J$ is
a maximal element, then $J\setminus\{j\}$ is downward closed, and we
have
\[
L_J(i) = L_{J\setminus\{j\}}(i) \underset{L(j)}\times D_1(j);
\]
in particular $L_J(i) \to L_{J\setminus\{j\}}(i)$ is proper and
surjective. By induction on $|J|$, we conclude that $L(i) \to
D_0(i)$ and hence $D_1(i) \to D_0(i)$ is proper and surjective.
Finally, $D_1(i)$ is a regular, quasi-projective $k$-scheme by
construction.
\end{proof}
We note that the proof above uses nothing special about schemes, and
it generalises without difficulty to the categorical setting of
\cite[exp.~V$^{\text{bis}}$,~5.1.4]{SGA4II}.

\begin{Cor}\label{Cor diagram of hypercoverings}
Let $(\mathscr J, \leq)$ be a poset such that for every $i \in
\mathscr J$ there are only finitely many $j \in \mathscr J$ with $j
\leq i$. Let $D \colon \mathscr J\op \to \Sch_k$ be a diagram of
separated finite type $k$-schemes. Then there exists a diagram
$D_\bullet \colon \mathscr J\op \times \Delta_+\op \to \Sch_k$ such
that the following hold.
\begin{enumerate}
\fancyitem{(\arabic{enumi})} $D_{-1} = D$;
\fancyitem{(\arabic{enumi})} For each $i \in \mathscr J$, the
diagram $D_\bullet(i)$ is a proper hypercovering of $D(i)$;
\fancyitem{(\arabic{enumi})} For each $i \in \mathscr J$ and each $n
\in \Z_{\geq 0}$, the scheme $D_n(i)$ is quasi-projective over $k$
and regular.
\end{enumerate}
\end{Cor}

\begin{proof}
Apply the procedure of
\cite[exp.~V$^{\text{bis}}$,~5.1.4--5.1.7,~5.2.4]{SGA4II} (see also
\cite[6.2.8]{Hdg3}), replacing $k$-schemes by $\mathscr J$-indexed
diagrams of $k$-schemes, and resolution of singularities (or
alterations) by Lemma \ref{Lem resolution of diagram}.
\end{proof}

\begin{Lemma}\label{Lem rank of connecting map}
Let $0 \to A^\bullet \to B^\bullet \to C^\bullet \to 0$ be a short
exact sequence of chain complexes of finite dimensional vector
spaces. Then
\[
\rk\Big(\delta^i \colon H^i(C^\bullet) \to H^{i+1}(A^\bullet)\Big) =
\rk\big(d_B^i\big) - \rk\big(d_A^i\big) - \rk\big(d_C^i\big).
\]
\end{Lemma}

\begin{proof}
All ranks in question only depend on the stupid truncations
$\sigma_{\geq i}\sigma_{\leq i+1}$ of the complexes. For $d_A^i$,
$d_B^i$, and $d_C^i$ this is clear, and for $\delta^i$ this follows
because of the factorisation
\[
H^i(\sigma_{\geq i} C^\bullet) \twoheadrightarrow H^i(C^\bullet)
\stackrel{\delta^i}\to H^{i+1}(A^\bullet) \hookrightarrow
H^{i+1}(\sigma_{\leq i+1} A^\bullet),
\]
noting that precomposing by surjections and postcomposing by
injections does not alter ranks. Now the snake lemma gives a long
exact sequence
\[0 \to \ker d_A^i \to \ker d_B^i \to \ker d_C^i \stackrel{\delta^i}\to \coker d_A^i \to \coker d_B^i \to \coker d_C^i \to 0.
\]
Additivity of dimension in short exact sequences gives
\begin{equation}
\dim(\ker d_A^i) - \dim(\ker d_B^i) + \dim(\ker d_C^i) -
\rk(\delta^i) = 0.\label{Eq rank 1}
\end{equation}
On the other hand, exactness of $0 \to A^\bullet \to B^\bullet \to
C^\bullet \to 0$ gives
\begin{equation}
\dim A^i - \dim B^i + \dim C^i = 0.\label{Eq rank 2}
\end{equation}
Subtracting \eqref{Eq rank 1} from \eqref{Eq rank 2} gives the
result.
\end{proof}

\begin{Cor}\label{Cor rank of map on cohomology}
Let $f \colon A^\bullet \to B^\bullet$ be a morphism of chain
complexes of finite dimensional vector spaces. Then
\[
\rk\left(H^i(f)\right) = \rk\left(\left(\begin{smallmatrix}d_A^i & 0
\\ f^i & -d_B^{i-1} \end{smallmatrix}\right) \colon A^i \oplus
B^{i-1} \to A^{i+1} \oplus B^i\right) - \rk\left(d_A^i\right) -
\rk\left(d_B^{i-1}\right).
\]
\end{Cor}

\begin{proof}
Apply Lemma \ref{Lem rank of connecting map} to the short exact
sequence
\[
0 \to B^\bullet[-1] \to C^\bullet(f) \to A^\bullet \to 0,
\]
noting that the boundary homomorphism of this sequence is $H^i(f)$.
\end{proof}

\begin{Thm}\label{Thm dimension}
Assume $k = \bar\F_p$. If $\Rk(X,Y)$ holds for all smooth projective
$k$-schemes $X$ and $Y$, then
\begin{enumerate}
\fancyitem{(\arabic{enumi})}\label{Item Thm Dim} $\Dim\cs(X)$ holds
for every separated finite type $k$-scheme $X$;
\fancyitem{(\arabic{enumi})}\label{Item Thm Rk} $\Rk\cs(X,Y)$ holds
for all separated finite type $k$-schemes $X$ and $Y$.
\end{enumerate}
\end{Thm}

\begin{proof}
Let $X$ and $Y$ be separated $k$-schemes of finite type, and let $f
\colon Y \to X$ be a proper morphism. Choose a commutative diagram
as in Lemma \ref{Lem compactification}
\begin{equation}
\begin{tikzcd}
Y \ar[open]{r}\ar{d}[swap]{f} & \bar Y \ar{d}{\bar f} & W \ar[closed']{l}[swap]{w}\ar{d}{g}\\
X \ar[open]{r} & \bar X & V \ar[closed']{l}{v}\punct{,}
\end{tikzcd}\label{Dia open and closed}
\end{equation}
where $\bar X$ and $\bar Y$ are compatible compactifications of $X$
and $Y$ respectively such that $X$ (resp.~$Y$) is dense in $\bar X$
(resp.~$\bar Y$) with complement $V$ (resp.~$W$), and $g$ denotes
$\bar f|_W$. Applying Corollary \ref{Cor diagram of hypercoverings}
to the right hand square of \eqref{Dia open and closed}, we get a
commutative diagram
\begin{equation}
\begin{tikzcd}
W_\bullet \ar{r}{w_\bullet}\ar{d}[swap]{g_\bullet} & \bar Y_\bullet \ar{d}{\bar f_\bullet} \\
V_\bullet \ar{r}{v_\bullet} & \bar X_\bullet
\end{tikzcd}\label{Dia hypercovering}
\end{equation}
of augmented simplicial schemes such that each $W_i$ is smooth
projective for~$i \geq 0$ and $W_\bullet$ is a proper hypercover of
$W_{-1} = W$, and similarly for $V$, $\bar X$, and $\bar Y$.

By Axiom \ref{Def additional axioms} \ref{Ax compactly supported
cohomology}, \ref{Ax mapping cone}, and \ref{Ax comparison of long
exact sequence}, we get a commutative square of long exact sequences
\begin{equation}\label{Dia commutative square of long exact sequences}
\begin{tikzcd}[row sep=.6em,column sep=1.9em,crossing over clearance=1ex]
 & \makebox*{}[l]{$\cdots$} \ar[shorten=1.5em]{rrr} & & & \mathclap{H^i\cs(Y)} \ar[shorten=1.5em]{rrr}\ar{dd} & & & \mathclap{H^i(\bar Y)} \ar[shorten=1.5em]{rrr}\ar{dd} & & & \mathclap{H^i(W)} \ar[shorten=1.5em]{rr}\ar{dd} & & \makebox*{}[r]{$\cdots$}\\
\makebox*{}[l]{$\cdots$} \ar[crossing over,shorten=1.5em]{rr} & & \mathclap{H^i\cs(X)} \ar[crossing over,shorten=1.5em]{rrr}\ar[start anchor=center, end anchor=center, shorten=2em]{rru} & & & \mathclap{H^i(\bar X)} \ar[crossing over,shorten=1.5em]{rrr}\ar[start anchor=center, end anchor=center, shorten=2em]{rru} & & & \mathclap{H^i(V)} \ar[crossing over,shorten=1.5em]{rrr}\ar[start anchor=center, end anchor=center, shorten=2em]{rru} & & & \makebox*{}[r]{$\cdots$} & \\
 & \makebox*{}[l]{$\cdots$} \ar[end anchor={[xshift=-.8em]},shorten=1.5em]{rrr} & & & \mathclap{H^i(C_\bullet(w))} \ar[start anchor={[xshift=.8em]},shorten=1.5em]{rrr} & & & \mathclap{H^i(\bar Y_\bullet)} \ar[shorten=1.5em]{rrr} & & & \mathclap{H^i(W_\bullet)} \ar[shorten=1.5em]{rr} & & \makebox*{}[r]{$\cdots$}\\
\makebox*{}[l]{$\cdots$} \ar[end
anchor={[xshift=-.75em]},shorten=1.5em]{rr} & &
\mathclap{H^i(C_\bullet(v))} \ar[start
anchor={[xshift=.75em]},shorten=1.5em]{rrr}\ar[from=uu, crossing
over]\ar[start anchor=center, end anchor=center, shorten=2em]{rru} &
& & \mathclap{H^i(\bar X_\bullet)}
\ar[shorten=1.5em]{rrr}\ar[from=uu, crossing over]\ar[start
anchor=center, end anchor=center, shorten=2em]{rru} & & &
\mathclap{H^i(V_\bullet)} \ar[shorten=1.5em]{rrr}\ar[from=uu,
crossing over]\ar[start anchor=center, end anchor=center,
shorten=2em]{rru} & & & \makebox*{}[r]{$\cdots$} &
\end{tikzcd}
\end{equation}
whose vertical maps are isomorphisms. Thus $H^i\cs(X)$ is computed
by $H^i(C_\bullet(v))$, which by \ref{Ax cohomology of simplicial
motive} is computed by the hypercohomology spectral sequence
\begin{equation}\label{Eq spectral sequence}
E_1^{p,q} = H^q(C_p(v)) \Ra H^{p+q}(C_\bullet(v)).
\end{equation}
If $\bar X_p$ and $V_p$ are defined over some finite field $k_0$ for
all $p \leq i+1$, then the computation of $H^i(C_\bullet(v))$ only
involves maps between cohomology groups of smooth projective
varieties defined over $k_0$. Moreover, the action of the
$|k_0|$-power geometric Frobenius (as an algebraic cycle) on
$H^q(C_p(v))$ is pure of weight~$q$ \cite[Corollary~1(2)]{KM} (see also
Corollary \ref{Cor characteristic polynomial}).

Since $\Ext^i_{K[x]}(M,N) = 0$ for all $i$ and for all
$K[x]$-modules $M$ and $N$ that are pure of different weights, the
spectral sequence \eqref{Eq spectral sequence} degenerates on the
$E_2$ page and the filtration on $E_\infty$ canonically splits:
\begin{equation}
H^i\cs(X) = H^i(C_\bullet(v)) \cong \bigoplus_{p+q=i}
E_2^{p,q},\label{Eq canonical splitting}
\end{equation}
where $E_2^{p,q}$ is given by
\[
E_2^{p,q} = \frac{\ker\Big(d^p \colon H^q(C_p(v)) \to
H^q(C_{p+1}(v))\Big)}{\im \Big(d^{p-1} \colon H^q(C_{p-1}(v)) \to
H^q(C_p(v))\Big)}.
\]
A dimension count gives
\[
\dim E_2^{p,q} = \dim\big(H^q(C_p(v))\big) - \rk(d^p) -
\rk(d^{p-1}).
\]
By assumption and by Remark \ref{Rmk Rk motives}, each of these
numbers is independent of the Weil cohomology theory $H$. Then the
same holds for $\dim E_2^{p,q}$ and therefore also for $\dim
H^i\cs(X)$, which proves \ref{Item Thm Dim} for $X$.

By \ref{Ax cohomology of simplicial motive}, the natural map
$f_\bullet \colon C_\bullet(v) \to C_\bullet(w)$ of cosimplicial
all-or-nothing motives induces a morphism of spectral sequences
\[
f_{\bullet,*} \colon E_*^{p,q}(X_\bullet) \to E_*^{p,q}(Y_\bullet).
\]
On the $E_1$ pages, this gives commutative diagrams
\begin{equation*}
\begin{tikzcd}[column sep=1.5em]
\ldots \ar{r} & H^q(C_p(v)) \ar{r}{d^p_v}\ar{d}{f_{p,*}} & H^q(C_{p+1}(v)) \ar{r}\ar{d}{f_{p+1,*}} & \ldots \\
\ldots \ar{r} & H^q(C_p(w)) \ar{r}{d^p_w} & H^q(C_{p+1}(w)) \ar{r} &
\ldots\punct{.}
\end{tikzcd}
\end{equation*}
By Corollary \ref{Cor rank of map on cohomology}, we get
\begin{equation}
\rk E_2^{p,q}(f_{\bullet,*}) = \rk\left(\begin{smallmatrix}d^p_v & 0
\\ f_{p,*} & -d^{p-1}_w\end{smallmatrix}\right) -
\rk\left(d^p_v\right) - \rk\left(d^{p-1}_w\right).\label{Eq rank
E_2}
\end{equation}
By assumption and by Remark \ref{Rmk Rk motives}, the right hand
side is independent of the Weil cohomology theory $H$, hence so is
the left hand side. Then the same holds for the rank of
$f_{\bullet,*} \colon H^i(C_\bullet(v)) \to H^i(C_\bullet(w))$ since
$f_{\bullet,*}$ respects the canonical splitting of \eqref{Eq
canonical splitting}, which proves \ref{Item Thm Rk} by diagram
\eqref{Dia commutative square of long exact sequences}.
\end{proof}

\begin{Rmk}\label{Rmk combine proofs}
Instead of the diagrammatic argument given above, one would be
tempted to use the strong version of the alterations result
\cite{dJ}.

This gives a proper hypercover $X_\bullet \to X$ along with an
embedding $X_\bullet \to \bar X_\bullet$ such that each $\bar X_n$
is smooth projective, and the complement of $X_n \subseteq \bar X_n$
is a simple normal crossings divisor $D_n$.

Assuming $\Rk(Y,Z)$ holds for smooth projective $k$-schemes $Y$ and
$Z$, a simplicial argument for $D_n$ shows that the dimension of
$H^i\cs(X_n)$ (hence also $H^i(X_n)$) is independent of the Weil
cohomology theory \cite[p.~29]{Katz}. This again uses the Weil
conjectures to conclude degeneration of a spectral sequence.

Then the spectral sequence for the hypercovering $X_\bullet \to X$
computes $H^i(X)$ in terms of $H^i(X_n)$. However, now the purity
argument no longer applies, and we have no idea on what page the
spectral sequence might degenerate. So even knowing $\Rk(Y,Z)$ for
smooth quasi-projective $k$-schemes $Y$ and $Z$ does not imply
$\Dim\cs(X)$ (or its variant $\Dim(X)$ for cohomology $H^i$) through
this method.

The above argument is a way around this problem. The author is not
aware of a place in the literature where this argument is carried
out, but variants of it might have been known to experts.
\end{Rmk}

\begin{Rmk}\label{Rmk WZ}
To prove Theorem \ref{Thm dimension} \ref{Item Thm Dim} and
\ref{Item Thm Rk} only in the case where $X$ and $Y$ are proper, one
can carry out the proof above without talking about (simplicial)
mapping cones. Indeed, applying the proof above to the morphism
$f_\bullet \colon Y_\bullet \to X_\bullet$ (of simplicial schemes)
instead of the morphism $f_\bullet \colon C_\bullet(v) \to
C_\bullet(w)$ (of cosimplicial all-or-nothing motives) shows that
the dimensions of $H^i(X)$ and $H^i(Y)$ as well as the rank of $f^*
\colon H^i(X) \to H^i(Y)$ do not depend on $H$.

Applying this to the closed immersion $v \colon V \to \bar X$ as in
\eqref{Dia open and closed}, we see that
\begin{align*}
\dim H^i\cs(X) &= \dim\Big(\ker H^i(v)\Big) + \dim\Big(\coker H^{i-1}(v)\Big) \\
&= \dim H^i\big(\bar X\big) - \rk\big(H^i(v)\big) + \dim
H^{i-1}\big(V\big) - \rk\big(H^{i-1}(v)\big),
\end{align*}
so finally $\dim H^i\cs(X)$ does not depend on $H$. This recovers
Theorem \ref{Thm dimension}~\ref{Item Thm Dim} without ever using
simplicial mapping cones. (The mapping cone of complexes is still
implicitly used through Corollary \ref{Cor rank of map on
cohomology}.)

One can also obtain Theorem \ref{Thm dimension}~\ref{Item Thm Rk} in
this way (at least for \'etale or crystalline cohomology), by
looking at the different weight parts as in \eqref{Eq canonical
splitting} and applying Corollary \ref{Cor rank of map on
cohomology} in the case where the complexes themselves are mapping
cones.

This requires a different kind of axiom comparable to \ref{Ax
comparison of long exact sequence}, using the mapping cone of
complexes instead. However, mapping cones in the derived category
are only unique up to non-unique isomorphism, and to get a
commutative diagram like \eqref{Dia commutative square of long exact
sequences} one needs to choose mapping cones functorially. We use
the ones from simplicial schemes, but there are other possibilities.
For example, Cisinki--D\'eglise's \emph{mixed Weil cohomologies}
\cite{CisDeg1} work in the dg setting, which may well be enough to
make this part of the argument work.
\end{Rmk}

\section{Dimensions and cycle class maps}\label{Sec Dimensions and cycle class maps}
In this section, the ground field $k$ is allowed to be arbitrary
again. We start with a Bertini irreducibility theorem. We do some
extra work in Lemma \ref{Lem Bertini containing subvariety} to avoid
extending the base field, using the Bertini irreducibility theorem
of Charles--Poonen \cite{ChaPoo} as well as the classical one
\cite[Theorem~6.3(4)]{JouBer}.

The main application of these Bertini theorems is Corollary \ref{Cor
difference of irreducible cycles}, which we use to prove the
implication $\text{(3b)} \Ra \text{(1)}$ of Theorem \ref{Thm
equivalent}. The idea is that if $Z \subseteq X$ is an effective
codimension $m$ cycle on a smooth projective variety $(X,H)$, then
for $m$ general sections $H_1, \ldots, H_m \in |H|$ containing $Z$,
the intersection $H_1 \cap \ldots \cap H_m$ contains only one new
component $Z'$, which is smooth away from $Z$ (in particular
reduced).
%
Therefore,
\[
[Z] = H^m - [Z'] \in \CH^m(X),
\]
which realises $[Z]$ as a difference of two irreducible cycles. We
do something similar for an arbitrary (not necessarily effective)
cycle $\alpha \in \CH^m(X)$.

We suggest the reader skip ahead to Corollary \ref{Cor difference of
irreducible cycles} on a first reading.

\begin{Def}
Let $X$ be a separated $k$-scheme of finite type. Consider the
following condition on a closed subscheme $Z \subseteq X$:
{\setitemize{leftmargin=*, align=parleft, labelsep=1cm}
\begin{itemize}
\item[$(\Irr)$:] For every irreducible component $X_i$ of $X_{\bar k}$, there exists a unique irreducible component $Z_i$ of $Z_{\bar k}$ contained in $X_i$, and moreover $Z_i = X_i \cap Z_{\bar k}$.
\end{itemize}}

\smallskip\noindent
If $X$ is geometrically normal, then $X_i \cap X_j = \varnothing$
for $i \neq j$. In this case, the final statement  of $(\Irr)$ is
automatic: we clearly have $Z_i \subseteq X_i \cap Z_{\bar k}$, and
all other components $Z_j$ of $Z_{\bar k}$ are disjoint from $X_i
\cap Z_{\bar k}$.
\end{Def}


\begin{Lemma}\label{Lem Bertini containing subvariety}
Let $(X,H)$ be a projective $k$-scheme all of whose components have
dimension $n \geq 2$, let $Z \subseteq X$ be a geometrically reduced
closed subscheme of pure dimension $\ell \neq 0,n$, and let $Y_1,
\ldots, Y_s \subseteq X$ be integral subschemes of dimension $< n$
that are not contained in $Z$. Assume that $X$ is smooth away from
$Z$ and at the generic points of $Z$. Then for $d \gg 0$, there
exists an element $D \in |dH|$ containing $Z$ such that
\begin{itemize}
\item $D$ is smooth away from $Z$;
\item $D$ is smooth at the generic points of $Z$;
\item the divisor $D \setminus Z$ of $X \setminus Z$ satisfies $(\Irr)$;
\item $D$ does not contain any of the $Y_j$.
\end{itemize}
\end{Lemma}

\begin{proof}
Write $X\sing$ and $Z\sing$ for the singular (non-smooth) loci of
$X \to \Spec k$ and $Z \to \Spec k$
respectively. By assumption, $X \setminus X\sing$ contains $U = X
\setminus Z$, as well as the generic points of $Z$. Since $Z$ is
geometrically reduced, $Z\sing$ has dense open complement. Thus, the
open subset $W = X \setminus (X\sing \cup Z\sing)$ is smooth of
dimension $n$, contains $U$, and $W \cap Z$ is smooth and dense. Let
$Y$ be the union of the zero-dimensional $Y_j$, and note that $Y
\cap Z = \varnothing$.

If $k$ is infinite, choose $d \gg 0$ such that $\mathcal I_Z(dH)$ is
globally generated, and let $V$ be the image of $H^0(X,\mathcal
I_Z(dH))$ in $H^0(X,\mathcal O_X(dH))$, i.e.\ $|V|$ is the linear
system in $|dH|$ of sections vanishing on $Z$. Since $\mathcal
I_Z(dH)$ is globally generated, the base locus of $|V|$ is exactly
$Z$, so $|V|$ defines a morphism
\[
\phi \colon U \to \P\big(V^*\big).
\]
Note that $\phi$ is a composition of a locally closed immersion and
a coordinate projection, hence unramified. Thus
\cite[Corollary~I.6.11(2,3)]{JouBer} shows that the locus of $D \in |V| =
\P(V)(k)$ such that $D \cap U$ is smooth and satisfies $(\Irr)$ is
dense open.

Since the $Y_j$ are not contained in $Z$, the locus of $D \in |V|$
not containing $Y_j$ is dense open. Similarly, the locus of $s \in
H^0(X,\mathcal I_Z(dH))$ whose restriction to $\mathcal I_Z/\mathcal
I_Z^2(dH)$ is nonzero is dense open. Thus we can find a $k$-point of
$\P(V)$ satisfying all desired properties, since a dense open subset
of $\P(V)$ has a $k$-point if $k$ is infinite.

If $k$ is finite, we apply \cite[Theorem~2.1]{Wutz}, where Wutz's
$X,Y,Z,k,\ell,m$ are our $W, Y, Z, \ell-1, \ell, n$ respectively.
This shows that the set
\[
\left\{D \in |dH|\ \Bigg|\ \begin{array}{l} Z \subseteq D \text{ and
} Y \cap D = \varnothing, \\ \dim (D \cap W)_{\operatorname{sing}}
\leq \ell - 1,  \\ D \cap (X \setminus Z) \text{ is smooth of
dimension } n-1. \end{array}\right\}
\]
has positive density $\mu > 0$ as $d \to \infty$. 
%
In particular, such $D$ are smooth at the generic points of $Z$,
since $Z$ has (pure) dimension $\ell$. Applying Bertini's
irreducibility theorem \cite[Theorem~1.2]{ChaPoo} to $X \setminus Z$
shows that the set of $D$ such that $D \setminus Z$ satisfies
$(\Irr)$ in $X \setminus Z$ has density $1$.

Finally, if $Y_j$ is positive-dimensional, then the ideal sheaf
$\mathcal I_{Z \cap Y_j} \subseteq \mathcal O_{Y_j}$ is nonzero
because $Y_j \not\subseteq Z$. The map $\mathcal I_Z \to \mathcal
I_{Z \cap Y_j}$ is surjective by the second and third isomorphisms
theorems.
Hence, for $d \gg 0$, the map
\[
\phi_d \colon H^0(X,\mathcal I_Z(d)) \to H^0(Y_j,\mathcal I_{Z \cap
Y_j}(d))
\]
is surjective. The dimension of the right hand side is (eventually)
a polynomial of degree $\dim Y_j > 0$ in $d$, so
\[
\codim \ker (\phi_d) \to \infty \hspace{2em}\text{ as } d \to
\infty.
\]
Hence, the functions that vanish on $Y_j$ have density $0$ as $d \to
\infty$. Therefore, the intersection of the three sets has positive
density $\mu$.
\end{proof}

\begin{Lemma}\label{Lem Bertini linkage}
Let $k$ be a perfect field, and let $(X,H)$ be a smooth projective
$k$-scheme all of whose components have dimension $n$. Let $Z_1,
\ldots, Z_r \subseteq X$ and $Y_1,\ldots, Y_s \subseteq X$ be
pairwise distinct integral subschemes of codimension $m \neq 0,n$.
%
Then for $d_1,\ldots,d_m \gg 0$, there exist sections
$D_1,\ldots,D_m$ of $|d_1H|,\ldots,|d_mH|$ intersecting properly
such that \vspace{-.3em}
\[
\left[\bigcap_{i=1}^m D_i\right] = \sum_{i = 1}^r [Z_i] + [Z],
\]
where $Z \subseteq X$ is (geometrically) reduced, satisfies
$(\Irr)$, does not contain any of the $Z_i$ and $Y_j$, and $Z
\setminus \bigcup Z_i$ is smooth.
\end{Lemma}

\begin{proof}
Let $Z' = \bigcup Z_i$. We apply Lemma \ref{Lem Bertini containing
subvariety} inductively on the codimension $m$ of the $Z_i$ to find
sections $D_i$ of $|d_iH|$ containing $Z'$ such that $\bigcap D_i$
is smooth away from $Z'$ and at the generic points of $Z'$ and does
not contain any of the $Y_j$, and $\bigcap D_i \setminus Z'$
satisfies $(\Irr)$. Thus the multiplicity of $\bigcap D_i$ at the
generic points of $Z'$ is $1$, so
\[
\left[\bigcap_{i=1}^m D_i\right] = \sum_{i = 1}^r [Z_i] + \alpha
\]
for an effective cycle $\alpha$ none of whose components is
contained in $Z'$.

If $Z$ is the closure of $\bigcap D_i \setminus Z'$, then
restricting to $X \setminus Z'$, we find that $\alpha = [Z]$, where
all coefficients are $1$ since $\bigcap D_i \setminus Z'$ is smooth.
Then $Z$ is geometrically reduced by \cite[Tag
\href{https://stacks.math.columbia.edu/tag/020I}{020I}]{Stacks}.
Finally, $Z$ satisfies $(\Irr)$ because $\bigcap D_i \setminus Z'$
does, and it does not contain any of the $Y_j$ since $\bigcap D_i
\setminus Z'$ does not contain any $Y_j \setminus Z'$ by the choice
of the $D_i$.
\end{proof}

\begin{Thm}\label{Thm difference of irreducible cycles}
Let $k$ be a perfect field, let $(X,H)$ be a smooth projective
$k$-scheme of equidimension $n$, let $\alpha \in \CH^m(X)$ be a pure
dimensional cycle, and let $e_0 \in \Z$. If $m \neq 0,n$, then there
exists a (geometrically) reduced subscheme $Z \subseteq X$
satisfying $(\Irr)$ and $e \geq e_0$ such that
\[
\alpha = [Z] - eH^m \in \CH^m(X).
\]
\end{Thm}

\begin{proof}
Write $\alpha = \sum_i n_i[Z_i] - \sum_i n_i'[Z_i'] - eH^m$, where
$Z_i, Z_i' \subseteq X$ are pairwise distinct integral subschemes of
codimension $m$, and $n_i, n_i' \in \Z_{\geq 0}$ (to start with, we
may take $e = 0$). We will apply Lemma \ref{Lem Bertini linkage} a
few times.

First, by induction on $z = \sum (n_i - 1) + \sum (n'_i - 1)$, we
will reduce to the case where $z = 0$. Indeed, if $z > 0$, then one
of the $n_i$ or $n_i'$ is $\geq 2$. Say $n_1 \geq 2$; the case $n'_i
\geq 2$ is similar. Applying Lemma \ref{Lem Bertini linkage} to
$Z_1$ with $\{Y_j\} := \{Z_i\ |\ i > 1\} \cup \{Z_i'\}$, we can
write $[Z_1] = dH^m - [Z']$, where $Z'$ does not contain any of the
$Z_i$ and $Z_i'$, and $Z'$ is generically smooth (in particular
reduced).
%
Adjoining the irreducible components of $Z'$ to the $Z_i'$ and
changing $e$ to $e-d$, we have reduced $z$ by one, because the new
components coming from $Z'$ all have coefficient $1$. After finitely
many steps, we get $z = 0$, so all $n_i$ and $n'_i$ are equal to
$1$.

Now applying Lemma \ref{Lem Bertini linkage} to the $Z_i$ while
avoiding the $Z'_i$, we get a generically smooth subscheme $Z'
\subseteq X$ of codimension $m$ and such that $\sum_i [Z_i] = d H^m
- [Z']$. Adjoining the components of $Z'$ to the $Z_i'$ and
replacing $e$ by $e-d$, we can write
\[
\alpha = - \sum_i [Z_i'] - eH^m.
\]
Finally, applying Lemma \ref{Lem Bertini linkage} to the $Z_i'$, for
every $d_1,\ldots,d_m \gg 0$ we get a geometrically reduced
subscheme $Z$ of codimension $m$ satisfying $(\Irr)$ such that $d
H^m = \sum_i [Z_i'] + [Z]$, where $d = \prod d_i$. Thus,
\[
\alpha = [Z] - (e+d)H^m.
\]
Choosing the $d_i$ large enough so that $e+d \geq e_0$ gives the
result.
\end{proof}

\begin{Cor}\label{Cor difference of irreducible cycles}
Let $k$ be a perfect field, let $(X,H)$ be a smooth projective
$k$-scheme of equidimension $n$, and let $\alpha \in \CH^m(X)$. If
$m \neq 0, n$, then there exist geometrically reduced subschemes
$Z_1,Z_2 \subseteq X$ satisfying $(\Irr)$ such that
\[
\alpha = [Z_1] - [Z_2] \in \CH^m(X).
\]
\end{Cor}

\begin{proof}
By Theorem \ref{Thm difference of irreducible cycles}, we may write
$\alpha = [Z_1] - eH^m$, where $Z_1$ satisfies $(\Irr)$ and $e$ may
be taken arbitrarily large. Applying the usual Bertini
irreducibility theorem \cite[Theorem~6.10(4)]{JouBer},
\cite[Corollary~1.4]{ChaPoo} inductively, we find a subscheme $Z_2
\subseteq X$ satisfying $(\Irr)$ with $[Z_2] = eH^m$. Indeed, over
an infinite field we can do this for any $e$
\cite[Theorem~6.10(4)]{JouBer}, whereas over a finite field the
positivity of the density \cite[Corollary~1.4]{ChaPoo} shows that there
exists $e_0$ such that for all $e \geq e_0$ we can find a member
satisfying $(\Irr)$.
\end{proof}

\begin{Thm}\label{Thm cycle class map}
Let $k$ be an algebraically closed field. Let $X$ be a smooth
projective $k$-scheme, and assume that $\Dim\cs(U)$ holds for every
open subscheme $U \subseteq X$. Then $\Cl(X)$ holds.
\end{Thm}

\begin{proof}
Since $\CH^*_\Q(U \amalg V) = \CH^*_\Q(U) \times \CH^*_\Q(V)$, and
the same statement holds for the cohomology ring, we may assume $X$
is irreducible of dimension $n$, hence (geometrically) integral.
Let $\alpha \in \CH^*_\Q(X)$ be given. Because the cycle class map
is homogeneous, it suffices to treat the case where $\alpha \in
\CH^m_\Q(X)$ is of pure dimension $d = n-m$. If $m = 0$ or $m = n$,
then clearly the kernels of the cycle class maps to $H^m(X)$ and
$\H^m(X)$ agree.

Now assume $m \neq 0,n$. Then by Corollary \ref{Cor difference of
irreducible cycles}, we may write $\alpha = [Z_1] - [Z_2]$, where
$Z_1$ and $Z_2$ are reduced subschemes satisfying $(\Irr)$; since
$X$ is integral this just means that $Z_1$ and $Z_2$ are integral as
well. If $\alpha = 0$ there is nothing to prove, so we may assume
$[Z_1] \neq [Z_2]$.  Let $Z = Z_1 \cup Z_2 \subseteq X$, let $U = X
\setminus Z$, and consider the long exact cohomology sequence with
compact support of \ref{Ax compactly supported cohomology}:
\[
\ldots \to H^i\cs(U) \to H^i\cs(X) \to H^i\cs(Z) \to \ldots.
\]
By \ref{Ax vanishing}, we have $\dim H\cs^{2d}(Z) = 2$ and
$H\cs^i(Z) = 0$ for $i > 2d$. Thus, we have an exact sequence
\begin{equation}
\ldots \to H\cs^{2d}(X) \stackrel{i^*}\to H\cs^{2d}(Z) \to
H\cs^{2d+1}(U) \to H\cs^{2d+1}(X) \to 0.\label{Eq long exact
sequence U}
\end{equation}
By \ref{Ax vanishing} again, the map $i^* \colon H\cs^{2d}(X) \to
H^{2d}\cs(Z) \cong K \oplus K$ is the dual of
\begin{align*}
K \oplus K &\rA H^{2m}(X)\\
(\lambda,\mu) &\longmapsto \lambda\cl(Z_1) + \mu\cl(Z_2).
\end{align*}
If $h$ is an ample divisor class on $X$, then $h^d \cdot Z_i > 0$.
Therefore, $\cl(Z_i) \neq 0$, so the rank of $i^*$ is either $1$ or
$2$. Additivity of dimensions in \eqref{Eq long exact sequence U}
gives
\[
\rk(i^*) = \dim H\cs^{2d}(Z) - \dim H\cs^{2d+1}(U) + \dim
H\cs^{2d+1}(X).
\]
By assumption, $\dim H^{2d+1}\cs(U)$ and $\dim H^{2d+1}\cs(X)$ are
independent of the Weil cohomology theory $H$, hence so is the rank
of $i^*$.

We now claim that $\cl_H(\alpha) = 0$ if and only if $\rk(i^*) = 1$
and $h^d \cdot Z_1 = h^d \cdot Z_2$. Indeed, if $\cl(\alpha) = 0$,
then $(i^*)^\vee$ has a kernel, so $i^*$ cannot have rank $2$.
Moreover, cupping the relation $\cl(Z_1) = \cl(Z_2)$ with $h^d$
gives $h^d \cdot Z_1 = h^d \cdot Z_2$. Conversely, if $\rk(i^*) = 1$
and $h^d \cdot Z_1 = h^d \cdot Z_2$, then there is a unique
$[\lambda_H:\mu_H] \in \P^1(K)$ such that
\[
\lambda_H\cl_H(Z_1) = \mu_H\cl_H(Z_2).
\]
Again, cupping with $h^d$ gives $\lambda_H Z_1 \cdot h^d = \mu_H Z_2
\cdot h^d$, forcing $[\lambda_H:\mu_H] = [1:1]$, so that
$\cl_H(\alpha) = 0$.

Because the rank of $i^*$ and the intersection numbers $h^d \cdot
Z_i$ are independent of the Weil cohomology theory $H$, this shows
that the vanishing of $\cl_H(\alpha)$ is also independent of $H$.
\end{proof}

\begin{Cor}
Let $k$ be an algebraically closed field. If $\Dim_c(X)$ holds for
every smooth quasi-projective $k$-scheme $X$, then $\Cl(X)$ holds
for every smooth projective $k$-scheme $X$. \hfill \qedsymbol
\end{Cor}


\end{document}